\documentclass[reqno]{amsart}
\usepackage{amsmath, amsfonts, amsthm, amssymb, setspace, textcomp, bbm, multirow, tikz, xstring}
\usepackage{geometry}
\newtheorem{theorem}{Theorem}[section]
\newtheorem{lemma}[theorem]{Lemma}

\newtheorem{corollary}[theorem]{Corollary}

\theoremstyle{definition}
\newtheorem{definition}[theorem]{Definition}
\newtheorem*{example*}{Example}

\theoremstyle{remark}

\numberwithin{equation}{section}

\newcommand{\hspt}[3]{\mathrm{spt}^{#1}_{#2}\hspace{-.2em} \left(#3\right) }  
\newcommand{\spt}[1]{\mathrm{spt}\left(#1\right) }  
\newcommand{\aqprod}[3]{\left(#1;#2\right)_{#3}}
\newcommand{\aprod}[2]{\left(#1\right)_{#2}}

\newcommand{\X}[1]{
    \IfEqCase{#1}{
        {38}{X1}
        {39}{X2}
        {40}{X6}
        {41}{X3}
        {42}{X4}
        {46}{X5}
    }[\PackageError{X}{Undefined option to X: #1}{}]%
}

\def\xOffSetForTikzPicture{30}
\def\yOffSetForTikzPicture{5}

\author{CATHERINE BABECKI}
\address{The Pennsylvania State University\\
State College, Pennsylvania 16802, USA
\endgraf cmb6625@psu.edu}

\author{CHRIS JENNINGS-SHAFFER}
\address{Department of Mathematics, Oregon State University\\
Corvallis, Oregon 97331, USA
\endgraf cjenningsshaffer@ufl.edu}

\author{GEOFFREY SANGSTON}
\address{University of Miami\\
Coral Gables, Florida 33146, USA
\endgraf geoffreysangston@gmail.com}

\thanks{This research was supported by the National Science Foundation Grant DMS-1359173.}



\title{Higher Order Smallest Parts Functions and Rank-Crank Moment Inequalities from Bailey Pairs}

\allowdisplaybreaks
\begin{document}

\allowdisplaybreaks

\begin{abstract}
We generalize a result of Garvan on inequalities and interpretations of the 
moments of the partition rank and crank functions. In particular for nearly 30 
Bailey pairs, we introduce a rank-like function, establish inequalities with 
the moments of the rank-like function and an associated crank-like function, 
and give an associated so called higher order smallest parts function. In some 
cases we are able to deduce inequalities among the rank-like functions.
We also conjecture additional inequalities and
a large number of congruences for the higher order smallest parts functions.
\end{abstract}

\maketitle

\allowdisplaybreaks


\section{Introduction}

The celebrated Rogers-Ramanujan identities state that
\begin{align*}
	\sum_{n=0}^\infty 
	\frac{q^{n^2}}{\aprod{q}{n}}
	&=
		\frac{1}{\aqprod{q}{q^5}{\infty} \aqprod{q^4}{q^5}{\infty} }	
	,&
	\sum_{n=0}^\infty 
	\frac{q^{n^2+n}}{\aprod{q}{n}}	
	&=
		\frac{1}{\aqprod{q^2}{q^5}{\infty}\aqprod{q^3}{q^5}{\infty} }	
	,
\end{align*}
where here and throughout the article 
\begin{align*}
	(a)_n 
	&:= 
		(a;q)_n = \prod_{j=0}^{n-1}(1-aq^j)
	,&
	(a)_\infty 
	&:=
		(a;q)_\infty = \prod_{j=0}^{\infty}(1-aq^j)
	.
\end{align*}
While these identities were stated and proved by Rogers in \cite{Rogers1},
they did not gain attention until they were rediscovered by Ramanujan
two decades later. It would be impossible to give here an adequate account of how
the Rogers-Ramanujan identities have found their way into various branches of
mathematics and related sciences, so we direct the reader to 
\cite{Andrews1, Andrews2, AndrewsBaxter1, AndrewsBerndt1, Askey1, Berndt1, 
GriffinOnoOle1,  LepowskyWilson1}.
What we do mention about these identities is that Rogers actually had several more 
identities of this type, 
and perhaps most important is that to give uniform proofs of 
such identities Bailey \cite{Bailey1,Bailey2} introduced what would later be 
known as the Bailey pair machinery.
In particular, we recall that a pair of sequences $(\alpha,\beta)$ form a
Bailey pair relative to $(a,q)$ if
\begin{align*}
	\beta_n
	&=
	\sum_{k=0}^n
	\frac{\alpha_k}{\aqprod{q}{q}{n-k}\aqprod{aq}{q}{n+k}}
	,
\end{align*}
and a limiting form of Bailey's Lemma is that
\begin{align*}
	\sum_{n=0}^\infty
	\aprod{x}{n} \aprod{y}{n}\left(\frac{aq}{xy}\right)^n
	\beta_n
	&=
		\frac{\aprod{\frac{aq}{x}}{\infty} \aprod{\frac{aq}{y}}{\infty}}
		{\aprod{aq}{\infty} \aprod{\frac{aq}{xy}}{\infty}}
		\sum_{n=0}^\infty
		\frac{\aprod{x}{n} \aprod{y}{n} \left(\frac{aq}{xy}\right)^n \alpha_n}
		{	\aprod{\frac{aq}{x}}{n} \aprod{\frac{aq}{y}}{n} }
	.
\end{align*}
By letting $x,y\rightarrow\infty$ in Bailey's Lemma and plugging in the 
Bailey pair, relative to $(a,q)$, given by
\begin{align*}
	\beta_n
	&=
	\frac{1}{\aprod{q}{n}}
	,&
	\alpha_n 
	&=
	\frac{(-1)^n (1-aq^{2n}) \aprod{a}{n} a^n q^{\frac{n(3n-1)}{2}} }
	{(1-a)\aprod{q}{n}}
,
\end{align*}
one sees that the right hand side sums to the appropriate products 
for $a=1$ and $a=q$, according to the Jacobi triple product identity,
so that Bailey has given a uniform proof of both Rogers-Ramanujan identities.
After Bailey's success with the above method,
Slater \cite{Slater1,Slater2} demonstrated the incredible power of Bailey pairs 
by giving over 100 identities of the Rogers-Ramanujan type by further
introducing Bailey pairs where an appropriate choice of $x$ and $y$ would allow 
the right hand side of Bailey's Lemma to sum to an infinite product.

Besides identities of the Rogers-Ramanujan type, Bailey pairs have found 
numerous uses in the study of $q$-series and integer partitions. We discuss
two recent uses related to counting the number of smallest parts in integer
partitions. For this we recall that a partition of a positive integer $n$
is a non-increasing sequence of positive integers that sum to $n$; we agree that
there is a single partition of $0$, which is the empty partition. We let
$p(n)$ denote the number of partitions of $n$. 
As an example, we see $p(5)=7$ as the seven partitions of $5$ are given by
$5$, $4+1$, $3+2$, $3+1+1$, $2+2+1$, $2+1+1+1$, and $1+1+1+1+1$.
The rank of a partition is given as the largest part minus the number of parts.
The crank of a partition is defined as the largest part, if the partition does
not contain any ones, and otherwise is the number of parts larger than the 
number of ones minus the number of ones.
Of the partitions of $5$ listed previously we see their respective ranks are
$4$, $2$, $1$, $0$, $-1$, $-2$, and $-4$, whereas their respective cranks are
$5$, $0$, $3$, $-1$, $1$, $-3$, and $-5$. We let $N(m,n)$ denote the number of
partitions of $n$ with rank $m$ and let $M(m,n)$ denote the number of partitions
of $n$ with crank $m$. We call the respective generating functions $R(z,q)$ and
$C(z,q)$, that is to say,
\begin{align*}
	R(z,q)
	&=
		\sum_{n=0}^\infty \sum_{m=-\infty}^\infty
		N(m,n)z^mq^n
	,&
	C(z,q)
	&=
		\sum_{n=0}^\infty \sum_{m=-\infty}^\infty
		M(m,n)z^mq^n
	.	
\end{align*}

In \cite{Andrews3}, Andrews introduced the function $\spt{n}$, which counts the
total number of appearances of the smallest part in each partition of $n$. We 
can think of this as a weighted count on the partitions of $n$, where the weight
is the number of times the smallest part appears. From the partitions of $5$
listed above, we see that $\spt{5}=14$. Given a relation with the second 
moment of the rank function, 
$\left(z\frac{\partial}{\partial z}\right)^2 R(z,q) \big|_{z=1}$,
Andrews proved that 
$\spt{5n+4}\equiv 0\pmod{5}$, $\spt{7n+5}\equiv 0\pmod{7}$, and 
$\spt{13n+6}\equiv 0\pmod{13}$. These congruences are reminiscent of 
Ramanujan's congruences for the partition function
$p(5n+4)\equiv 0\pmod{5}$, $p(7n+5)\equiv 0\pmod{7}$, and 
$p(11n+6)\equiv 0\pmod{11}$. 
One can easily verify that a generating function for $\spt{n}$ is given by
\begin{align*}
	S(q)
	&=
		\sum_{n=0}^\infty
		\spt{n}q^n
	=
		\sum_{n=1}^\infty
		\frac{q^n}{(1-q^n)^2 \aprod{q^{n+1}}{\infty}}
.
\end{align*}
With this in mind Andrews, Garvan, and Liang \cite{AndrewsGarvanLiang1} 
introduced a so called spt-crank as the series
\begin{align*}
	S(z,q)
	&=
		\sum_{n=0}^\infty
		\sum_{m=-\infty}^\infty
		N_{S}(m,n)z^mq^n
	=
		\sum_{n=1}^\infty
		\frac{q^n \aprod{q^{n+1}}{\infty} } 
		{ \aprod{zq^n}{\infty}\aprod{z^{-1}q^n}{\infty} }			
	.
\end{align*}
It is trivial to see that $S(1,q)=S(q)$. Three facts that are not so trivial to
prove are that each $N_{S}(m,n)$ is a non-negative integer, the 
$5$-dissection of $S(e^{2\pi i/5},q)$ gives another proof of the 
congruence $\spt{5n+4}\equiv 0\pmod{5}$, and the 
$7$-dissection of $S(e^{2\pi i/7},q)$ gives another proof of the 
congruence $\spt{7n+5}\equiv 0\pmod{7}$. The essential trick for the
dissections is to take the following Bailey pair relative to $(1,q)$,
\begin{align*}
	\beta_n
	&=
	\frac{1}{\aprod{q}{n}}
	,&
	\alpha_n 
	&=
	\left\{\begin{array}{ll}
		1 												& \mbox{ if } n=0,
		\\
		(-1)^n (1+q^{n}) q^{\frac{n(3n-1)}{2}}  	&\mbox{ if } n\ge 1,
	\end{array}\right.
\end{align*}
and apply Bailey's Lemma with $x=z$ and $y=z^{-1}$ to deduce that
\begin{align*}
	(1-z)(1-z^{-1})S(z,q)
	&=
		R(z,q)-C(z,q)
	.
\end{align*}
Inspired by this identity, Garvan and the second author \cite{GarvanJenningsShaffer1} 
investigated 
smallest parts functions related to overpartitions and partitions without 
repeated odd parts in terms of spt-cranks and differences between ranks and
cranks. Again the essential trick was to apply 
Bailey's Lemma with $x=z$ and $y=z^{-1}$, but with a different Bailey pair. 
In particular, there we used the Bailey pairs
\begin{align*}
	\beta_n
	&=
		\frac{1}{\aqprod{q^2}{q^2}{n}}
	,&
	\alpha_n 
	&=
		\left\{\begin{array}{ll}
		1 							& \mbox{ if } n=0,
		\\
		(-1)^n 2 q^{n^2}  			&\mbox{ if } n\ge 1,
		\end{array}\right.
	\\	
	\beta_n
	&=
		\frac{1}{2}
		\left(\frac{1}{\aqprod{q^2}{q^2}{n}}+\frac{(-1)^n}{\aqprod{q^2}{q^2}{n}}\right)
	,&
	\alpha_n 
	&=
		\left\{\begin{array}{ll}
		1 							& \mbox{ if } n=0,
		\\
		(-1)^n (1+q^{n^2})  		&\mbox{ if } n\ge 1,
		\end{array}\right.
	\\
	\beta_n
	&=
		\frac{1}{\aqprod{-q,q^2}{q^2}{n}}
	,&
	\alpha_n 
	&=
		\left\{\begin{array}{ll}
		1 								& \mbox{ if } n=0,
		\\
		(-1)^n 2 q^{2n^2-n}(1+q^{2n})  	&\mbox{ if } n\ge 1,
		\end{array}\right.
\end{align*}
of which the first two are relative to $(1,q)$ and the third is relative to 
$(1,q^2)$.
Here we note that
the same form of Bailey's Lemma gave four identities for spt-crank functions.
It is then natural to ask what would happen if one was to look at all of
Slater's Bailey pairs in this framework of spt-cranks; Garvan and the second author 
carried this out in a series of articles 
\cite{GarvanJenningsShaffer2, JenningsShaffer1, JenningsShaffer2}.
This can be compared to Slater's work, but instead of choosing Bailey pairs 
that would result in a series that sums to a product by the Jacobi triple
product identity, we needed to choose Bailey pairs with $a=1$ and such that
the associated spt-crank-type function dissected nicely at roots of unity.
Altogether this process resulted in over 20 spt-crank-type functions 
and associated spt-type functions with congruences.

Another use of Bailey's Lemma related to smallest parts functions
arose in \cite{Garvan1}, where Garvan considered the ordinary and symmetrized 
moments of the rank and crank functions,
\begin{align*}
	N_k(n)
	&= 
		\sum_{m=-\infty}^{\infty}m^k N(m,n)
	,&
	\eta_{k}(n)
	&= 
		\sum_{m = -\infty}^{\infty}
		\binom{m+\lfloor\frac{k-1}{2}\rfloor}{k}N(m,n)
	,\\	
	M_k(n)
	&= 
		\sum_{m=-\infty}^{\infty}m^kM(m,n)
	,&
	\mu_{k}(n)
	&= 
		\sum_{m = -\infty}^{\infty}
		\binom{m+\lfloor\frac{k-1}{2}\rfloor}{k}M(m,n)
	.
\end{align*}
Studies of the ordinary moments of the rank and crank function began with
the work of Atkin and Garvan in \cite{AtkinGarvan1}, and Andrews introduced 
$\eta_{2k}$, the symmetrized moment of the rank, in \cite{Andrews4}.
Previous to Garvan's article, it was conjectured that
$M_{2k}(n)>N_{2k}(n)$ for all positive $k$ and $n$ (here only the even moments
are of interest as the odd moments are zero). By asymptotics this was known to
hold for sufficiently large $n$ for each $k$ \cite{BringmannMahlburgRhoades1}.
Among Garvan's results, we highlight three. 
The first is the following form of the generating function of 
$\mu_{2k}(n)-\eta_{2k}(n)$, 
\begin{align*}
	\sum_{n=1}^\infty
	\left(\mu_{2k}(n)-\eta_{2k}(n)\right)q^n	
	&=
	\sum_{n_k\ge\dotsb\ge n_1\ge 1}
	\frac{q^{n_1+\dotsb+n_k}}{\aprod{q^{n_1+1}}{\infty}(1-q^{n_1})^2\dotsm(1-q^{n_k})^2}
	,
\end{align*}
which clearly exhibits that $\mu_{2k}(n)\ge \eta_{2k}(n)$ for all $k$ and $n$, 
and additionally one can easily determine when the inequality is strict.
The second is a formula for writing the ordinary moments as a positive integer 
linear combination of the symmetrized moments, and in particular 
$M_{2k}(n)-N_{2k}(n) \ge \mu_{2}(n)-\eta_{2}(n)$.
The last is a family of weighted counts of the partitions of
$n$, $\hspt{}{k}{n}$, the higher order smallest parts function, 
such that the weighting is clearly non-negative and based
on the frequency of the parts of the partitions,
$\hspt{}{1}{n}=\spt{n}$, and
$\hspt{}{k}{n}=\mu_{2k}(n)-\eta_{2k}(n)$. Additionally Garvan established a large
number of congruences for $\hspt{}{k}{n}$.
The linchpin in establishing the generating function for 
$\mu_{2k}(n)-\eta_{2k}(n)$
was again a certain form of Bailey's Lemma applied to the Bailey pair
\begin{align*}
	\beta_n
	&=
	\frac{1}{\aprod{q}{n}}
	,&
	\alpha_n 
	&=
	\left\{\begin{array}{ll}
		1 												& \mbox{ if } n=0,
		\\
		(-1)^n (1+q^{n}) q^{\frac{n(3n-1)}{2}}  	&\mbox{ if } n\ge 1.
	\end{array}\right.
\end{align*}
Inspired by Garvan's results, in \cite{JenningsShaffer3}
the  second author carried out the same study of
differences and inequalities between rank and crank moments
related to overpartitions and 
partitions without repeated odd parts. The upshot was that one needed only
choose different Bailey pairs compared to Garvan's work.

It is now obvious what we are to do next. We are to consider the framework 
developed by Garvan in \cite{Garvan1}, but applied to all applicable Bailey 
pairs of Slater. Again this compares with Slater's original use of Bailey pairs. 
As we will see shortly, the requirements for 
our choice of Bailey pairs are just that $a=1$; $\alpha_0=\beta_0=1$; formulaically 
$\alpha_{-n}=\alpha_n$; and after multiplying by an infinite product, of our 
own choice, it is clear that $\beta_n$ has non-negative coefficients.
Our results will mirror that of Garvan's study of rank and crank moments.
For each Bailey pair considered, we will introduce a rank and crank-like 
function, obtain a generating function for the difference of symmetrized
moments which clearly exhibits non-negative coefficients,
deduce an inequality for the associated ordinary moments, and then give 
a weighted count of partitions that agrees with the difference of the
symmetrized moments.
To demonstrate that this can be applied to Bailey pairs past those in Slater's 
list, we also consider one Bailey pair from \cite{BowmannMclaughlinSills1}. 
Altogether we will give 28 instances of this process.

The rest of the article is organized as follows. In Section 2 we give our
definitions and main results, which are series identities and inequalities. 
In Section 3 we prove the series identities
and inequalities listed in Section 2. In Section 4 we prove the combinatorial
interpretation of the symmetrized rank and crank moment differences,
which justifies our definitions of higher order spt functions. In Section 5
we end with a few conjectures and remarks.

\section{Definitions and Statement of Main Results}

For our series identities and inequalities, we need a small number of general
identities, all of which are straightforward to prove. These identities have their
origins in a combination of classical works on the rank function
as well as \cite{Andrews4, Garvan1, JenningsShaffer3}, however here we state 
and prove them in generality. We combine the main identities into
the single following theorem. Given that the proofs primarily already exist in
the literature, we will find it takes far longer to state our results than to
prove them.

\begin{theorem}\label{TheoremMain}
Suppose $(\alpha,\beta)$ is a Bailey pair relative to $(1,q)$,
$\alpha_0=\beta_0=1$, and $\alpha_n=\alpha_{-n}$.
Here we note that by $\alpha_n=\alpha_{-n}$, we are treating $\alpha_n$ as a 
bilateral sequence by extending $\alpha_n$ to negative indices according to whatever 
general formula is given for $\alpha_n$.
Suppose
\begin{align*}
	R_X(z,q)
	:=	
	P_{X}(q)
	\left(
		1
		+
		\sum_{n = 1}^{\infty} \frac{\alpha_{n}q^n(1-z)(1-z^{-1})}{(1-zq^n)(1-z^{-1}q^n)}
	\right)	
	&=
	\sum_{n=0}^\infty \sum_{m=-\infty}^\infty
		N_{X}(m,n) z^mq^n
	,
\end{align*}
where $P_{X}(q)$ is a series in $q$,
and for $k$ a positive integer let
\begin{align*}
	N^{X}_{k}(n)
	&=
		\sum_{m = -\infty}^{\infty}
		m^kN_{X}(m,n)
	,&
	\eta^{X}_{k}(n)
	&= 
		\sum_{m = -\infty}^{\infty}
		\binom{m+\lfloor\frac{k-1}{2}\rfloor}{k}N_{X}(m,n)
	.
\end{align*}
Then
\begin{align}\label{EqMainTheoremSymmetrizedRankMoment}
	\sum_{n=1}^\infty \eta^X_{2k}(n) q^n
	&=
		-P_X(q)
		\sum_{n=1}^\infty
		\frac{\alpha_{n}q^{nk}}{(1-q^n)^{2k}}
	,
\end{align}	
and
\begin{align}\label{EqMainTheoremSymmetrizexRankCrankDifference}		
	\sum_{n=1}^\infty
	\hspt{X}{k}{n}q^n
	&:=
		P_X(q)\aprod{q}{\infty}
		\sum_{n=1}^\infty \mu_{2k}(n)q^n
		-
		\sum_{n=1}^\infty \eta^{X}_{2k}(n)q^n
	\nonumber\\
	&=
		P_X(q)
		\sum_{n_k\ge \dotsb \ge n_1\ge 1 }
		\frac{\aprod{q}{n_1}^2 \beta_{n_1} q^{n_1+\dotsb +n_k} }
		{(1-q^{n_1})^2\dotsm(1-q^{n_k})^2}	
	.
\end{align}
 
Furthermore, $N^X_{k}(n)$ and $\eta^{X}_{k}(n)$ are zero if $k$ is odd,
and the moments for even $k$ are related by the identities
\begin{align}
	\label{EqMainTheoremSymmetrizedRankEquality}
	\eta^{X}_{2k}(n)
	&= 
		\frac{1}{(2k)!}\sum_{m=-\infty}^{\infty}g_{k}(m)N_{X}(m,n)
	,\\
	\label{EqMainTheoremOrdinaryRankEquality}
	N^{X}_{2k}(n)
	&= 
		\sum_{j=1}^{k}(2j)!S^*(k,j)\eta^{X}_{2j}(n)
	,
\end{align}
where $g_k(x)=\prod_{j=0}^{k-1}(x^2-j^2)$, and the sequence $S^*(k,j)$ is 
defined recursively by $S^*(k+1,j)=S^*(k,j-1)+j^2S^*(k,j)$ and boundary 
conditions $S^*(1,1)=1$ and $S^*(k,j)=0 $ if  $j\leq 0$ or $j>k$.
\end{theorem}

\begin{example*}
We first demonstrate the use of this theorem with a specific Bailey pair.
Consider the Bailey pair B(2) of \cite{Slater1},
\begin{align*} 
	\beta_n 
	&= 
		\frac{q^n}{(q)_n} 
	,&
	\alpha_n 
	&=
		\begin{cases} 
		1 								& n=0 
		,\\ 
		(-1)^nq^{\frac{3n(n-1)}{2}}(1+q^{3n}) 	& n\geq 1 
		.\\
	\end{cases}
\end{align*}
We note that using this formula with negative $n$ does give that 
$\alpha_{n} = \alpha_{-n}$.
Looking to Theorem \ref{TheoremMain}, if $\hspt{X}{k}{n}$ is to be a
non-negative integer, then we should choose 
$P_X(q)=\frac{1}{\aprod{q}{\infty}}$. From this we now know we should define a
rank-like function by
\begin{align*}
	R_{B2}(z,q)
	&= 
		\sum_{n=0}^\infty \sum_{m=-\infty}^\infty	
		N_{B2}(m,n)z^mq^n
	=
		\frac{1}{\aqprod{q}{q}{\infty}}
		\left(
			1
			+
			\sum\limits_{n=1}^\infty
			\frac{(1-z)(1-z^{-1})(-1)^{n}q^{\frac{n(3n-1)}{2}}(1+q^{3n})}{(1-zq^n)(1-z^{-1}q^n)}
		\right)
	,
\end{align*}	
and we have the associated ordinary and symmetrized moments given by
\begin{align*}
	N^{B2}_{k}(n)
	&=
		\sum_{m = -\infty}^{\infty}
		m^kN_{B2}(m,n)
	,&
	\eta^{B2}_{k}(n)
	&= 
		\sum_{m = -\infty}^{\infty}
		\binom{m+\lfloor\frac{k-1}{2}\rfloor}{k}N_{B2}(m,n)
	.
\end{align*}
But then 
\begin{align*}
	\sum_{n=1}\left(\mu_{2k}(n)-\eta^{B2}_{2k}(n)\right)q^n
	&=
		\sum_{n_k\ge \dotsb \ge n_1\ge 1 }
		\frac{q^{2n_1+n_2+\dotsb +n_k} }
		{\aprod{q^{n_1+1}}{\infty}(1-q^{n_1})^2\dotsm(1-q^{n_k})^2}	
	,
\end{align*}
and clearly $\mu_{2k}(n)\ge \eta^{B2}_{2k}(n)$ for positive $k$ and $n$.
Additionally by taking $k=1$ and examining the $n_1=1$ summand,
\begin{align*}
	\frac{q^2}{\aprod{q^2}{\infty}(1-q)^2}
	&=
		q^2 + 2q^3+ 4q^4 + 7q^5 + 12q^6 + 19q^7 + 30q^8 + \dotsb
,		   
\end{align*}
we see that $\mu_{2}(n)> \eta^{B2}_{2}(n)$ for $n\ge 2$.
To obtain an inequality for the ordinary moments, we make the following
observation. The $S^*(k,j)$ are non-negative integers and $S^*(k,1)$ is
positive for $k\ge 1$. In particular, we then have that
\begin{align*} 
	M_{2k}(n)-N^{B2}_{2k}(n) 
	&=
		\sum_{j=1}^{k}(2j)!S^*(k,j)( \mu_{2j}(n)-\eta^{B2}_{2j}(n))
	\ge 
		\mu_{2}(n)-\eta^{B2}_{2}(n)
.
\end{align*}
Thus $M_{2k}(n)>N^{B2}_{2k}(n)$ for all positive $k$ and $n\ge 2$.
Furthermore, we can ask what is the non-negative integer
$\hspt{B2}{k}{n}:= \mu_{2k}(n)- \eta^{B2}_{2k}(n)$ counting in terms of
partitions? This is the question we address in Section 4.
\end{example*}

We now repeat this process with the many relevant Bailey pairs from
\cite{Slater1,Slater2}, and tabulate our results in a corollary.
In the cases where a Bailey pair would have fractional powers of $q$, we 
replace $q$ with the appropriate power of $q$.
It is worth noting that the Bailey pairs $B1$, 
$E1$, and 
the unlabeled Bailey pair on page 468 of \cite{Slater1} with $\beta_n=\frac{1}{(-q^{1/2})_{n}\aprod{q}{n}}$, 
(which is also $F1$ with $q^{1/2}\mapsto-q^{1/2}$)
correspond respectively to
the ordinary rank studied by Garvan in \cite{Garvan1} 
and the Dyson rank for overpartitions 
and the M2-rank for partitions without repeated odd parts
studied by the second author studied in \cite{JenningsShaffer3}.
The M2-rank for overpartitions corresponds to a Bailey pair from a 
specialization of a finite form
of the Jacobi triple product identity. As such, 
we omit these Bailey pairs from our consideration.
In labeling our Bailey pairs, we use the existing label in the literature when
it exists, otherwise we label the Bailey pairs relative to $(1,q)$ as 
$X1$, $X2$, $X3$, $X4$, $X5$, $X6$ and
the Bailey pairs relative to $(1,q^2)$ as
$Y1$, $Y2$, $Y3$, $Y4$. This labeling is not meant to carry any additional
semantic value.

To state our corollary, we first introduce the relevant crank-like functions
that will appear. In the cases of a crank that has appeared before in the 
literature, we follow the existing naming conventions.
We let
\begin{align*}
	C(z,q)
	&=
		\frac{\aprod{q}{\infty}}{\aprod{zq}{\infty}\aprod{z^{-1}q}{\infty}}	
		=
		\sum_{n=0}^\infty\sum_{m=-\infty}^\infty
		M(m,n)z^mq^n
	,\\		
	\overline{C}(z,q)
	&=	
		\frac{\aprod{-q}{\infty}\aprod{q}{\infty}}{\aprod{zq}{\infty}\aprod{z^{-1}q}{\infty}}	
		=
		\sum_{n=0}^\infty\sum_{m=-\infty}^\infty
		\overline{M}(m,n)z^mq^n
	,\\
	C^J(z,q)	
	&=	
		\frac{\aprod{q}{\infty}}{\aqprod{q^3}{q^3}{\infty}\aprod{zq}{\infty}\aprod{z^{-1}q}{\infty}}	
		=
		\sum_{n=0}^\infty\sum_{m=-\infty}^\infty
		M^J(m,n)z^mq^n
	,\\
	C^{\X{40}}(z,q)
	&=
		\frac{\aprod{q}{\infty}}
		{\aqprod{q^2}{q^2}{\infty}\aprod{zq}{\infty}\aprod{z^{-1}q}{\infty}}	
		=
		\sum_{n=0}^\infty\sum_{m=-\infty}^\infty
		M^{\X{40}}(m,n)z^mq^n
	,\\
	C^F(z,q)	
	&=	
		\frac{\aqprod{q^2}{q^2}{\infty}}{\aqprod{zq^2}{q^2}{\infty}\aqprod{z^{-1}q^2}{q^2}{\infty}}	
		=
		\sum_{n=0}^\infty\sum_{m=-\infty}^\infty
		M^F(m,n)z^mq^n
	,\\
	C^G(z,q)
	&=
		\frac{\aqprod{-q}{q^2}{\infty}}
		{\aqprod{zq^2}{q^2}{\infty}\aqprod{z^{-1}q^2}{q^2}{\infty}}	
		=
		\sum_{n=0}^\infty\sum_{m=-\infty}^\infty
		M^G(m,n)z^mq^n
	,\\
	C^Y(z,q)
	&=
		\frac{1}
		{\aqprod{zq^2}{q^2}{\infty}\aqprod{z^{-1}q^2}{q^2}{\infty}}	
		=
		\sum_{n=0}^\infty\sum_{m=-\infty}^\infty
		M^Y(m,n)z^mq^n
	,\\
	C^{L2}(z,q)
	&=
		\frac{\aprod{-q}{\infty}\aqprod{q^4}{q^4}{\infty}}
		{\aqprod{zq^4}{q^4}{\infty}\aqprod{z^{-1}q^4}{q^4}{\infty}}	
		=
		\sum_{n=0}^\infty\sum_{m=-\infty}^\infty
		M^{L2}(m,n)z^mq^n		
	.
\end{align*}
We note that $C(z,q)$ is the ordinary crank of partitions, the moments of 
which Garvan studied in \cite{Garvan1}, and 
the function $\overline{C}(z,q)$ is known as the (first residual) crank of 
overpartitions \cite{BringmannLovejoyOsburn1}. 
Since all of
these functions are directly related to the ordinary crank, 
upon defining the moments in the obvious way, which
we omit, based on \cite{Garvan1} we have the following,
\begin{align*}
	\sum_{n=1}^\infty \mu_{2k}(n)q^n 
	&=
		\frac{1}{(q)_{\infty}}
		\sum\limits_{n=1}^{\infty} \frac{(-1)^{n+1}q^{\frac{n(n-1)}{2}+kn}(1+q^n)}{(1-q^n)^{2k}} 
	,\\
	\sum_{n=1}^\infty \overline{\mu}_{2k}(n)q^n 
	&=
		\frac{(-q)_{\infty}}{(q)_\infty}
		\sum_{n=1}^{\infty} 
		\frac{(-1)^{n+1}q^{\frac{n(n-1)}{2}+kn}(1+q^n)}{(1-q^n)^{2k}} 
	,\\
	\sum_{n=1}^\infty \mu^{J}_{2k}(n)q^n 
	&=
		\frac{1}{(q^3;q^3)_{\infty}(q)_\infty}
		\sum_{n=1}^{\infty} 
		\frac{(-1)^{n+1}q^{\frac{n(n-1)}{2}+kn}(1+q^n)}{(1-q^n)^{2k}}
	,\\
	\sum_{n=1}^\infty \mu^{\X{40}}_{2k}(n)q^n 
	&= 
		\frac{1}{(q)_\infty(q^2;q^2)_{\infty}}
		\sum\limits_{n=1}^{\infty} \frac{(-1)^{n+1}q^{\frac{n(n-1)}{2}+kn}(1+q^n)}{(1-q^n)^{2k}} 
	,\\
	\sum_{n=1}^\infty \mu^{F}_{2k}(n)q^n 
	&=
		\frac{1}{(q^2;q^2)_{\infty}}
		\sum_{n=1}^{\infty} 
		\frac{(-1)^{n+1}q^{n(n-1)+2kn}(1+q^{2n})}{(1-q^{2n})^{2k}}
	,\\
	\sum_{n=1}^\infty \mu^{G}_{2k}(n)q^n 
	&= 
		\frac{(-q;q^2)_\infty}{(q^2;q^2)^2_\infty}
		\sum\limits_{n=1}^\infty \frac{(-1)^{n+1}q^{n(n-1)+2kn}(1+q^{2n})}{(1-q^{2n})^{2k}}
	,\\
	\sum_{n=1}^\infty \mu^{Y}_{2k}(n)q^n 
	&= 
		\frac{1}{(q^2;q^2)^2_\infty}
		\sum\limits_{n=1}^\infty \frac{(-1)^{n+1}q^{n(n-1)+2kn}(1+q^{2n})}{(1-q^{2n})^{2k}}
	,\\
	\sum_{n=1}^\infty \mu^{L2}_{2k}(n)q^n 
	&=
		\frac{(-q)_\infty}{(q^4;q^4)_\infty} 
		\sum_{n=1}^\infty \frac{(-1)^{n+1}q^{2n(n-1)+4kn}(1+q^{4n})}{(1-q^{4n})^{2k}}
	.
\end{align*}
We find that the ordinary and symmetrized moments satisfy the same relation
as for the rank-like functions. In particular,
\begin{align*} 
	M^X_{2k}(n)
	&=
		\sum_{j=1}^{k}(2j)!S^*(k,j)\mu^X_{2j}(n).
\end{align*}
We note that if one wishes to actually compute 
$\mu^{X}_{2k}(n)-\eta^{X}_{2k}(n)$ numerically, one should do so with the above 
representation for $\mu^{X}_{2k}(n)$ and that of $\eta^{X}_{2k}(n)$
in (\ref{EqMainTheoremSymmetrizedRankMoment}), rather than
by (\ref{EqMainTheoremSymmetrizexRankCrankDifference}).

\begin{corollary}\label{CorollarySeriesIdentities}
(1) Using the Bailey pair A(1) from \cite{Slater1}, relative to $(1,q)$,
\begin{align*} 
	\beta_n 
	&= 
		\frac{1}{(q)_{2n}}
	,& 
 	\alpha_n 
 	&=
	\begin{cases} 
		1 						& n=0, 
		\\ 
		-q^{6k^2-5k+1} 			& n= 3k-1, 
		\\
		q^{6k^2-k}+q^{6k^2+k} 	&n= 3k, 
		\\
		-q^{6k^2+5k+1}  		& n=3k+1,
	\end{cases} 
\end{align*}
we define
\begin{align*}
	R_{A1}(z,q)
	&= 
		\frac{1}{\aprod{q}{\infty}}
		\left( 
			1
			-
			\sum\limits_{n=1}^\infty 
			\frac{(1-z)(1-z^{-1})q^{6n^2 - 2n}}{(1-zq^{3n-1})(1-z^{-1}q^{3n-1})} 
			-
			\sum\limits_{n=0}^\infty 
			\frac{(1-z)(1-z^{-1})q^{6n^2+8n + 2}}{(1-zq^{3n+1})(1-z^{-1}q^{3n+1})}   
			\right.\\&\left.\quad
			+
			\sum\limits_{n=1}^\infty 
			\frac{(1-z)(1-z^{-1})q^{6n^2+2n}(1+q^{2n}) }{(1-zq^{3n})(1-z^{-1}q^{3n})} 
		\right)
	,
\end{align*}
and obtain
\begin{align*}
	&\sum\limits_{n=1}^{\infty} \eta^{A1}_{2k}(n)q^{n} 
	= 
		\frac{1}{(q)_{\infty}}
		\left(
			\sum\limits_{n=1}^\infty \frac{q^{6n^2-5n+1+(3n-1)k}}{(1-q^{3n-1})^{2k}} 
			+
			\sum\limits_{n=0}^\infty \frac{q^{6n^2+5n+1+(3n+1)k}}{(1-q^{3n+1})^{2k}} 
			-
			\sum\limits_{n=1}^\infty \frac{q^{6n^2-n+3nk}(1+q^{2n})}{(1-q^{3n})^{2k}} 
		\right)
	,\\ 
	&\sum_{n = 1}^{\infty} \hspt{A1}{k}{n} q^{n} 
	:= 
		\sum_{n = 1}^{\infty}\left(\mu_{2k}(n)-\eta_{2k}^{A1}(n)\right)q^{n} 
		 = 
		 \sum\limits_{n_k \geq \cdots \geq n_1 \geq 1} 
		 \frac{q^{n_1+\cdots{}+n_k}} 
		 {(q^{n_1+1})_{n_1}(q^{n_1+1})_{\infty}(1-q^{n_k})^{2}\cdots{}(1-q^{n_1})^2}
	. 
\end{align*}

(2) Using the Bailey pair A(3) from \cite{Slater1}, relative to $(1,q)$,
\begin{align*} 
	\beta_n 
	&= 
		\frac{q^n}{(q)_{2n}} 
	,& 
	\alpha_n 
	&=
		\begin{cases} 
			1 							& n=0, 
			\\ 
			-q^{6k^2-2k} 				& n= 3k-1, 
			\\
			q^{6k^2-2k}+q^{6k^2+2k} 	& n= 3k, 
			\\
			-q^{6k^2+2k}  				& n=3k+1,
			\\
		\end{cases} 
\end{align*}
we define
\begin{align*}
	R_{A3}(z,q)
	&= 
		\frac{1}{\aprod{q}{\infty}}
		\Bigg(
			1
			-
			\sum\limits_{n=1}^\infty 
			\frac{(1-z)(1-z^{-1})q^{6n^2+n-1}}{(1-zq^{3n-1})(1-z^{-1}q^{3n-1})} 
			-
			\sum\limits_{n=0}^\infty 
			\frac{(1-z)(1-z^{-1})q^{6n^2+5n+1}}{(1-zq^{3n+1})(1-z^{-1}q^{3n+1})}   
			\\&\quad
			+
			\sum\limits_{n=1}^\infty 
			\frac{(1-z)(1-z^{-1})q^{6n^2+n}(1+q^{4n})}{(1-zq^{3n})(1-z^{-1}q^{3n})} 
		\Bigg)  
	,
\end{align*}
and obtain
\begin{align*}
	&\sum\limits_{n=1}^\infty \eta^{A3}_{2k}(n)q^n 
	=  
		\frac{1}{(q)_{\infty}}
		\left(
			\sum\limits_{n=1}^\infty \frac{q^{6n^2-2n+(3n-1)k}}{(1-q^{3n-1})^{2k}} 
			+
			\sum\limits_{n=0}^\infty \frac{q^{6n^2+2n+(3n+1)k}}{(1-q^{3n+1})^{2k}}   
			-
			\sum\limits_{n=1}^\infty \frac{q^{6n^2-2n+3nk}(1+q^{4n})}{(1-q^{3n})^{2k}} 
		\right)
	,\\
	&\sum_{n = 1}^{\infty} \hspt{A3}{k}{n}q^{n} 
	:= 
		\sum_{n = 1}^{\infty}\left(\mu_{2k}(n) - \eta_{2k}^{A3}(n)\right)q^{n} 
		=
		\sum\limits_{n_k \geq \dotsc \geq n_1 \geq 1} 
		\frac{q^{2n_1+n_2+\dotsb+n_k}} 
		{(q^{n_1+1})_{n_1}(q^{n_1+1})_{\infty}(1-q^{n_k})^2\dotsm(1-q^{n_1})^2} 
	.
\end{align*}
    
(3) Using the Bailey pair A(5) from \cite{Slater1}, relative to $(1,q)$,
\begin{align*} 
	\beta_n 
	&= 
		\frac{q^{n^2}}{(q)_{2n}} 
	,& 
	\alpha_n 
	&=
		\begin{cases} 
		1 						& n=0 
		,\\ 
		-q^{3k^2-k} 			& n= 3k-1 
		,\\
		q^{3k^2-k}+q^{3k^2+k} 	& n= 3k 
		,\\
		-q^{3k^2+k}  			& n=3k+1
		,\\
		\end{cases} 
\end{align*}
we define
\begin{align*}
	R_{A5}(z,q)
	&= 
		\frac{1}{\aprod{q}{\infty}}
		\Bigg( 
			1
			-
			\sum\limits_{n=1}^\infty 
			\frac{(1-z)(1-z^{-1})q^{3n^2+2n-1}}{(1-zq^{3n-1})(1-z^{-1}q^{3n-1})} 
			-
			\sum\limits_{n=0}^\infty 
			\frac{(1-z)(1-z^{-1})q^{3n^2+4n+1}}{(1-zq^{3n+1})(1-z^{-1}q^{3n+1})}   
			\\&\quad
			+
			\sum\limits_{n=1}^\infty 
			\frac{(1-z)(1-z^{-1})q^{3n^2+2n}(1+q^{2n})}{(1-zq^{3n})(1-z^{-1}q^{3n})} 
		\Bigg)  
	,	
\end{align*}
and obtain
\begin{align*}
	&\sum\limits_{n=1}^\infty \eta^{A5}_{2k}(n)q^n 
	=  
		\frac{1}{(q)_{\infty}}
		\left(
			\sum\limits_{n=1}^\infty \frac{q^{3n^2-n+(3n-1)k}}{(1-q^{3n-1})^{2k}} 
			+
			\sum\limits_{n=0}^\infty \frac{q^{3n^2+n+(3n+1)k}}{(1-q^{3n+1})^{2k}}   
			-
			\sum\limits_{n=1}^\infty \frac{q^{3n^2-n+3nk}(1+q^{2n})}{(1-q^{3n})^{2k}} 
		\right)
	,\\
	&
	\sum_{n = 1}^{\infty} \hspt{A5}{k}{n}q^{n} 
	:= 
		\sum_{n = 1}^{\infty}\left(\mu_{2k}(n)-\eta_{2k}^{A5}(n)\right)q^{n} 
	=
		\sum\limits_{n_k \geq  \dotsb \geq n_1 \geq 1} 
		\frac{q^{n_1^2+n_1+n_2+\dotsb+n_k}} 
		{(q^{n_1+1})_{n_1}(q^{n_1+1})_{\infty}(1-q^{n_k})^2\dotsm(1-q^{n_1})^2} 
	.
\end{align*}
    
(4) Using the Bailey pair A(7) from \cite{Slater1}, relative to $(1,q)$,
\begin{align*} 
	\beta_n 
	&= 
		\frac{q^{n^2-n}}{(q)_{2n}} 
	,&
 	\alpha_n 
 	&=
		\begin{cases} 
		1 							& n=0, \\ 
		-q^{3k^2-4k+1} 				& n=3k-1, \\
		q^{3k^2-2k}+q^{3k^2+2k} 	& n=3k, \\
		-q^{3k^2+4k+1}  			& n=3k+1,\\
		\end{cases}
\end{align*}
we define
\begin{align*}
	R_{A7}(z,q)
	&= 
		\frac{1}{\aprod{q}{\infty}}
		\Bigg( 
			1
			-\sum\limits_{n=1}^\infty 
			\frac{(1-z)(1-z^{-1})q^{3n^2-n}}{(1-zq^{3n-1})(1-z^{-1}q^{3n-1})} 
			-
			\sum\limits_{n=0}^\infty 
			\frac{(1-z)(1-z^{-1})q^{3n^2+7n+2}}{(1-zq^{3n+1})(1-z^{-1}q^{3n+1})}   
			\\&\quad
			+
			\sum\limits_{n=1}^\infty 
			\frac{(1-z)(1-z^{-1})q^{3n^2+n}(1+q^{4n})}{(1-zq^{3n})(1-z^{-1}q^{3n})} 
		\Bigg)  
	,	
\end{align*}
and obtain
\begin{align*}
	&\sum\limits_{n=1}^\infty \eta^{A7}_{2k}(n)q^n 
	= 
		\frac{1}{(q)_{\infty}}
		\left(
			\sum\limits_{n=1}^\infty \frac{q^{3n^2-4n+1+(3n-1)k}}{(1-q^{3n-1})^{2k}} 
			+
			\sum\limits_{n=0}^\infty \frac{q^{3n^2+4n+1+(3n+1)k}}{(1-q^{3n+1})^{2k}}   
			-
			\sum\limits_{n=1}^\infty \frac{q^{3n^2-2n+3nk}(1+q^{4n})}{(1-q^{3n})^{2k}} 
		\right)
	,\\
	&
	\sum_{n = 1}^{\infty} \hspt{A7}{k}{n}q^{n} 
	:= 
		\sum_{n = 1}^{\infty}\left(\mu_{2k}(n)-\eta_{2k}^{A7}(n)\right)q^{n} 
	= 
		\sum\limits_{n_k \geq \dotsb \geq n_1 \geq 1} 
		\frac{q^{n_1^2+n_2+\dotsb+n_k}} 
		{(q^{n_1+1})_{n_1}(q^{n_1+1})_{\infty}(1-q^{n_k})^2\dotsm (1-q^{n_1})^2}
	.
\end{align*}
    
(5) Using the Bailey pair B(2) from \cite{Slater1}, relative to $(1,q)$,
\begin{align*} 
	\beta_n &= \frac{q^n}{(q)_n} 
	,&
	\alpha_n 
	&=
		\begin{cases} 
		1 								& n=0 
		,\\ 
		(-1)^nq^{\frac{3n(n-1)}{2}}(1+q^{3n}) 	& n\geq 1 
		,\\
	\end{cases}
\end{align*}
we define
\begin{align*}
	R_{B2}(z,q)
	&= 
		\frac{1}{\aprod{q}{\infty}}
		\Bigg(
			1
			+
			\sum\limits_{n=1}^\infty
			\frac{(1-z)(1-z^{-1})(-1)^{n}q^{\frac{n(3n-1)}{2}}(1+q^{3n})}{(1-zq^n)(1-z^{-1}q^n)}
		\Bigg)
	,	
\end{align*}
and obtain
\begin{align*}
	&\sum\limits_{n=1}^\infty \eta^{B2}_{2k}(n)q^n 
	=
		\frac{1}{(q)_\infty}
		\sum_{n=1}^\infty 
		\frac{(-1)^{n+1}q^{\frac{3n(n-1)}{2}+nk}(1+q^{3n})}{(1-q^n)^{2k}}
	,\\
	&
	\sum_{n = 1}^{\infty} \hspt{B2}{k}{n}q^{n} 
	:= 
		\sum_{n = 1}^{\infty}\left(\mu_{2k}(n) -\eta_{2k}^{B2}(n)\right)q^{n} 
	= 
		\sum\limits_{n_k \geq \dotsb \geq n_1 \geq 1} 
		\frac{q^{2n_1+n_2+\dotsb+n_k}} 
		{(q^{n_1+1})_{\infty}(1-q^{n_k})^2\dotsm (1-q^{n_1})^2} 
	.
\end{align*}

(6) Using the Bailey pair C(1) from \cite{Slater1}, 
which is also L(6) from \cite{Slater2},
relative to $(1,q)$,
\begin{align*}
	\beta_{n} 
	&= 
		\frac{1}{\aqprod{q}{q^2}{n} \aprod{q}{n}} 
	,&
	\alpha_{n}
	&=
		\begin{cases} 
		1 									& n=0 
		,\\ 
		(-1)^{k}q^{3k^2 - k}(1 + q^{2k}) 	& n = 2k 
		,\\
		0 									& n = 2k + 1
		,
		\end{cases}
\end{align*}
we define
\begin{align*}
	R_{C1}(z,q) 
	&= 
		\frac{1}{\aprod{q}{\infty}}
		\Bigg(
			1 
			+ 
			\sum_{n = 1}^{\infty}
			\frac{(1-z)(1-z^{-1})(-1)^{n}q^{3n^2 + n}(1 + q^{2n})}{(1-zq^{2n})(1-z^{-1}q^{2n})}
		\Bigg)
	,
\end{align*}
and obtain
\begin{align*}	
	\sum_{n = 1}^{\infty}\eta_{2k}^{C1}(n)q^n 
	&= 
		\frac{1}{(q)_\infty}
		\sum_{n = 1}^{\infty}\frac{(-1)^{n+1}q^{3n^2 - n + 2nk}(1 + q^{2n})}{(1-q^{2n})^{2k}}
	,\\
	\sum_{n = 1}^{\infty} \hspt{C1}{k}{n}q^{n} 
	&:= 
		\sum_{n = 1}^{\infty}\left(\mu_{2k}(n)-\eta_{2k}^{C1}(n)\right)q^{n} 
    =  
    	\sum_{n_{k} \geq \dotsb \geq n_{1} \geq 1}
    	\frac{q^{n_{1} + \dotsb + n_{k}}}
    	{(q;q^2)_{n_{1}} (q^{n_{1} + 1})_{\infty} (1-q^{n_{k}})^{2}\dotsm (1-q^{n_{1}})^{2}}
	.
\end{align*}

(7) Using the Bailey pair C(2) from \cite{Slater1}, relative to $(1,q)$,
\begin{align*}
	\beta_{n} 
	&= 
		\frac{q^{n}}{\aqprod{q}{q^2}{n} \aprod{q}{n}} 
	,&
	\alpha_{n}
	&=
		\begin{cases} 
		1 										& n=0 
		,\\ 
		(-1)^{k}q^{3k^2 - k}(1 + q^{2k}) 		& n = 2k 
		,\\
		(-1)^{k+1}q^{3k^2 + k}(1 - q^{4k + 2}) 	& n = 2k + 1
		,
		\end{cases}
\end{align*}
we define
\begin{align*}
	R_{C2}(z,q)
	&= 
		\frac{1}{\aprod{q}{\infty}}
		\Bigg(
			1
			+
			\sum_{n=1}^{\infty}
			\frac{(1-z)(1-z^{-1})(-1)^{n}q^{3n^2 + n}(1 + q^{2n})}{(1-zq^{2n})(1-z^{-1}q^{2n})} 
			\\&\quad 
   			+
   			\sum_{n = 0}^{\infty}
   			\frac{(1-z)(1-z^{-1})(-1)^{n+1}q^{3n^2 + 3n + 1}(1 - q^{4n + 2})}{(1-zq^{2n+1})(1-z^{-1}q^{2n+1})} 
   		\Bigg)
	,
\end{align*}
and obtain
\begin{align*}
	\sum_{n = 1}^{\infty}\eta_{2k}^{C2}(n)q^n 
	&= 
		\frac{1}{(q)_\infty}
		\left(
			\sum_{n = 1}^{\infty}\frac{(-1)^{n+1}q^{3n^2 - n + 2nk}(1 + q^{2n})}{(1-q^{2n})^{2k}} 
			+ 
			\sum_{n = 0}^{\infty}\frac{(-1)^{n}q^{3n^2 + n + (2n+1)k}(1 - q^{4n + 2})}{(1-q^{2n + 1})^{2k}}
	\right)
	,\\
	\sum_{n = 1}^{\infty} \hspt{C2}{k}{n}q^{n} 
	&:= 
		\sum_{n = 1}^{\infty}\left(\mu_{2k}(n)-\eta_{2k}^{C2}(n)\right)q^{n} 
    = 
    	\sum_{n_{k} \geq \dotsb \geq n_{1} \geq 1}
    	\frac{q^{2n_{1} + n_{2} + \dotsb + n_{k}}}
    	{(q;q^2)_{n_{1}} (q^{n_{1} + 1})_{\infty}  (1-q^{n_{k}})^{2}\dotsm (1-q^{n_{1}})^{2}}
	.
\end{align*}

(8) Using the Bailey pair C(5) from \cite{Slater1}, 
which is also L(4) from \cite{Slater1},
relative to $(1,q)$,
\begin{align*}
	\beta_{n} 
	&= 
		\frac{q^{\frac{n(n-1)}{2}}}{\aqprod{q}{q^2}{n} \aprod{q}{n}} 
	,&
	\alpha_{n}
	&=
		\begin{cases} 
		1 									& n=0 
		,\\ 
		(-1)^{k}q^{k^2 - k}(1 + q^{2k}) 	& n = 2k
		,\\
		0 									& n = 2k + 1
		,
		\end{cases}
\end{align*}
we define
\begin{align*}
	R_{C5}(z,q)
	&= 
		\frac{1}{\aprod{q}{\infty}} 
		\Bigg(
			1
			+
			\sum_{n =1}^{\infty}
			\frac{(1-z)(1-z^{-1})(-1)^{n}q^{n^2+ n}(1 + q^{2n})}{(1-zq^{2n})(1-z^{-1}q^{2n})} 
		\Bigg) 
	,
\end{align*}
and obtain
\begin{align*}
	\sum_{n = 1}^{\infty}\eta_{2k}^{C5}(n)q^n 
	&= 
		\frac{1}{(q)_\infty}
		\sum_{n = 1}^{\infty}\frac{(-1)^{n+1}q^{n^2 - n + 2nk}(1 + q^{2n})}{(1-q^{2n})^{2k}}
	,\\
	\sum_{n = 1}^{\infty} \hspt{C5}{k}{n}q^{n} 
	&:= 
		\sum_{n = 1}^{\infty}\left(\mu_{2k}(n)-\eta_{2k}^{C5}(n)\right)q^{n} 
    = 
    	\sum_{n_{k} \geq  \dotsb \geq n_{1} \geq 1}
    	\frac{q^{ \frac{n_1(n_1 + 1)}{2} + n_{2} + \dotsb + n_{k}}}
    	{(q;q^2)_{n_{1}} (q^{n_{1} + 1})_{\infty}  (1-q^{n_{k}})^{2}\dotsm (1-q^{n_{1}})^{2}}
	.
\end{align*}

(9) Using the Bailey pair L(5), upon correcting the formula for $\beta_n$, from \cite{Slater2}, 
which also appears as the first entry in the second table of page 468 of \cite{Slater1},
relative to $(1,q)$,
\begin{align*}
	\beta_{n} 
	&= 
		\frac{(-1)_n}{(q)_{n}(q;q^2)_n}
	,& 
	\alpha_{n} 
	&=
		\begin{cases} 
		1 						& n=0
		,\\ 
		q^{\frac{n(n-1)}{2}}(1+q^n) 	& n \geq 1
		,		
		\end{cases}
\end{align*}
we define
\begin{align*}      
	R_{L5}(z,q)
	&=
		\frac{1}{\aprod{q}{\infty}}
		\Bigg(
			1 
			+ 
			\sum_{n = 1}^{\infty}\frac{(1-z)(1-z^{-1})q^{n(n+1)/2}(1+q^n)}{(1-zq^n)(1-z^{-1}q^n)}
		\Bigg)
	,
\end{align*}
and obtain
\begin{align*}
	&\sum\limits_{n=1}^\infty \eta^{L5}_{2k}(n)q^n 
	=
		\frac{-1}{(q)_\infty}
		\sum_{n=1}^\infty \frac{q^{\frac{n(n-1)}{2}+nk}(1+q^n)}{(1-q^{n})^{2k}} 
	,\\
	&
	\sum_{n = 1}^{\infty} \hspt{L5}{k}{n}q^{n} 
	:= 
		\sum_{n = 1}^{\infty}\left(\mu_{2k}(n) - \eta_{2k}^{L5}(n)\right)q^{n} 
	= 
		\sum\limits_{n_k \geq \dotsb \geq n_1 \geq 1} 
		\frac{(-1)_{n_1}q^{n_1+\dotsb +n_k}} 
		{(q;q^2)_{n_1} (q^{n_1+1})_{\infty}(1-q^{n_k})^2\dotsm (1-q^{n_1})^2} 
	.
\end{align*}

(10) Using the Bailey pair in the seventh entry in the table on page 470 of \cite{Slater1}, 
relative to $(1,q)$,
\begin{align*}
	\beta_{n}
	&= 
		\frac{\aqprod{-1}{q^2}{n}}{\aprod{q}{2n}} 
	,&
	\alpha_{n}
	&=
		\begin{cases}
		1 									& n = 0 
		,\\
		(-1)^{k}q^{2k^2 - k}(1 + q^{2k}) 	& n = 2k 
		,\\
		(-1)^{k}q^{2k^2 + k}(1 - q^{2k+1}) 	& n = 2k + 1
		,
		\end{cases}
\end{align*}
we define
\begin{align*}
	R_{\X{38}}(z,q) 
	&= 
		\frac{1}{\aprod{q}{\infty}} 
		\Bigg(
			1 
			+ 
			\sum_{n = 1}^{\infty}
			\frac{(1-z)(1-z^{-1})(-1)^{n}q^{2n^2 + n}(1 + q^{2n})}{(1 - zq^{2n})(1-z^{-1}q^{2n})} 
			\\&\quad
			+
			\sum_{n = 0}^{\infty}
			\frac{(1-z)(1-z^{-1})(-1)^{n}q^{2n^2 +3n+1}(1 - q^{2n+1})}{(1 - zq^{2n+1})(1 - z^{-1}q^{2n+1})}
		\Bigg)
	,
\end{align*}
and obtain
\begin{align*}
	&\sum_{n = 1}^{\infty}\eta_{2k}^{\X{38}}(n)q^n 
	= 
		\frac{1}{(q)_{\infty}} 
		\left(
			\sum_{n = 1}^{\infty} \frac{(-1)^{n+1}q^{2n^2 - n + 2nk}(1 + q^{2n})}{(1-q^{2n})^{2k}} 
			+
			\sum_{n = 0}^{\infty} \frac{(-1)^{n+1}q^{2n^2 + n + (2n+1)k}(1 - q^{2n + 1})}{(1 - q^{2n + 1})^{2k}}
		\right)
	,\\
	&\sum_{n = 1}^{\infty} \hspt{\X{38}}{k}{n}q^{n} 
	:= 
		\sum_{n = 1}^{\infty}\left(\mu_{2k}(n)-\eta_{2k}^{\X{38}}(n)\right)q^{n} 
	= 
		\sum_{n_{k} \geq \dotsb \geq n_{1} \geq 1} 
		\frac{(-1;q^{2})_{n_{1}}q^{n_{1} + \dotsb + n_{k}}}
		{(q^{n_{1} + 1})_{n_{1}} (q^{n_{1} + 1})_{\infty} (1-q^{n_{k}})^{2}\dotsm (1-q^{n_{1}})^{2}}
	.
\end{align*}

(11) Using the Bailey pair in the eighth entry in the table on page 470 of \cite{Slater1},
upon correcting the formula for $\alpha_n$, 
relative to $(1,q)$,
\begin{align*}
	\beta_{n} 
	&= 
		\frac{q^{n} \aqprod{-1}{q^2}{n}}{ \aprod{q}{2n}} 
	,&
	\alpha_{n} 
	&=
		\begin{cases}
		1 										& n = 0
		,\\
		(-1)^{k}q^{2k^2 - k}(1 + q^{2k}) 		& n = 2k 
		,\\
		(-1)^{k+1}q^{2k^2 + k}(1 - q^{2k+1}) 	& n = 2k + 1
		,
		\end{cases}
\end{align*}
we define
\begin{align*}
	R_{\X{39}}(z,q) 
	&= 
		\frac{1}{\aprod{q}{\infty}} 
		\Bigg(
			1 
			+ 
			\sum_{n = 1}^{\infty}
			\frac{(1-z)(1-z^{-1})(-1)^{n}q^{2n^2 + n}(1 + q^{2n})}{(1 - zq^{2n})(1-z^{-1}q^{2n})} 
			\\&\quad
			+
			\sum_{n = 0}^{\infty}
			\frac{(1-z)(1-z^{-1})(-1)^{n+1}q^{2n^2 +3n+1}(1 - q^{2n+1})}{(1 - zq^{2n+1})(1 - z^{-1}q^{2n+1})}
		\Bigg)
	,
\end{align*}
and obtain
\begin{align*}
	&\sum_{n = 1}^{\infty}\eta_{2k}^{\X{39}}(n)q^n 
	= 
		\frac{1}{(q)_{\infty}} 
		\left(
			\sum_{n = 1}^{\infty} \frac{(-1)^{n+1}q^{2n^2 - n + 2nk}(1 + q^{2n})}{(1-q^{2n})^{2k}}  
			+
			\sum_{n = 0}^{\infty} \frac{(-1)^{n}q^{2n^2 + n + (2n+1)k}(1 - q^{2n + 1})}{(1 - q^{2n + 1})^{2k}}
		\right)
	,\\
	&\sum_{n = 1}^{\infty} \hspt{\X{39}}{k}{n}q^{n} 
	:= 
		\sum_{n = 1}^{\infty}\left(\mu_{2k}(n)-\eta_{2k}^{\X{39}}(n)\right)q^{n} 
	= 
		\sum_{n_{k} \geq \dotsb \geq n_{1} \geq 1} 
		\frac{(-1;q^{2})_{n_{1}} q^{2n_{1} + n_{2} + \dotsb + n_{k}}}
		{(q^{n_{1} + 1})_{n_{1}} (q^{n_{1} + 1})_{\infty} (1-q^{n_{k}})^{2}
			\dotsm (1-q^{n_{1}})^{2}}
	.
\end{align*}

(12) Using the Bailey pair in the first entry in the table on page 471 of \cite{Slater1}, 
relative to $(1,q)$,
\begin{align*}
	\beta_{n} 
	&= 
		\begin{cases} 
		1 										& n=0 
		,\\ 
		\frac{\aqprod{-q^{2}}{q^{2}}{n-1}}
		{\aprod{q}{2n}} 					& n \geq 1
		,
		\end{cases} 
	&
	\alpha_{n} 
	&=
		\begin{cases}
		1 							& n = 0 
		,\\
		0 							& n = 4k-2 
		,\\
		-q^{8k^2 - 6k + 1} 			& n = 4k -1 
		,\\
		q^{8k^2 - 2k}(1 + q^{4k}) 	& n = 4k 
		,\\
		-q^{8k^2 + 6k + 1} 			& n = 4k + 1
		,\end{cases}
\end{align*}
we define
\begin{align*}
	R_{\X{41}}(z,q) 
	&= 
		\frac{1}{\aprod{q}{\infty}}
		\Bigg(
			1 
			- 
			\sum_{n = 1}^{\infty}
			\frac{(1-z)(1-z^{-1})q^{8n^2 - 2n}}{(1-zq^{4n-1})(1-z^{-1}q^{4n-1})} 
			+ 
			\sum_{n=1}^{\infty}
			\frac{(1-z)(1-z^{-1})q^{8n^2 + 2n}(1 + q^{4n})}{(1 - zq^{4n})(1 - z^{-1}q^{4n})} 
			\\&\quad
			- 
			\sum_{n =0}^{\infty}
			\frac{(1-z)(1-z^{-1})q^{8n^2 + 10n + 2}}{(1 -zq^{4n+1})(1 - z^{-1}q^{4n+1})}
		\Bigg)
	,
\end{align*}
and obtain
\begin{align*}
	&\sum_{n = 1}^{\infty}\eta_{2k}^{\X{41}}(n)q^n
	= 
		\frac{1}{(q)_{\infty}}
		\Bigg(
			\sum_{n = 1}^{\infty} \frac{q^{8n^2 - 6n + 1 + (4n-1)k}}{(1-q^{4n-1})^{2k}} 
			- 
			\sum_{n = 1}^{\infty} \frac{q^{8n^2 - 2n + 4nk}(1 + q^{4n})}{(1-q^{4n})^{2k}} 
			\\&\quad  
			+ 
			\sum_{n = 0}^{\infty} \frac{q^{8n^2 + 6n + 1 + (4n + 1)k}}{(1 - q^{4n + 1})^{2k}}
		\Bigg)
	,\\
	&\sum_{n = 1}^{\infty} \hspt{\X{41}}{k}{n}q^{n} 
	:= 
		\sum_{n = 1}^{\infty}\left(\mu_{2k}(n)-\eta_{2k}^{\X{41}}(n)\right)q^{n} 
	= 
		\sum_{n_{k} \geq \dotsb \geq n_{1} \geq 1} 
		\frac{(-q^2;q^2)_{n_{1}-1} q^{n_{1} + \dotsb + n_{k}}}
		{(q^{n_{1}+1})_{n_{1}} (q^{n_{1} + 1})_{\infty}
			(1-q^{n_{k}})^{2}\dotsm (1-q^{n_{1}})^{2}}
	.
\end{align*}

(13) Using the Bailey pair in the second entry in the table on page 471 of \cite{Slater1}, 
relative to $(1,q)$,
\begin{align*}
	\beta_{n} 
	&= 
		\begin{cases} 
		1 										& n=0 
		,\\ 
		\frac{q^{n} \aqprod{-q^2}{q^2}{n-1}}
		{\aprod{q}{2n}} 					& n \geq 1
		,
		\end{cases} 
	&
	\alpha_{n} 
	&=
		\begin{cases}
		1 								& n = 0 
		,\\
		0 								& n = 4k-2 
		,\\
		-q^{8k^2 - 2k} 					& n = 4k -1 
		,\\
		q^{8k^2 - 2k}(1 + q^{4k}) 		& n = 4k 
		,\\
		-q^{8k^2 + 2k} 					& n = 4k + 1
		,
		\end{cases}
\end{align*}
we define
\begin{align*}
	R_{\X{42}}(z,q) 
	&= 
		\frac{1}{\aprod{q}{\infty}}
		\Bigg(
			1 
			- 
			\sum_{n = 1}^{\infty}
			\frac{(1-z)(1-z^{-1})q^{8n^2 +2n-1}}{(1-zq^{4n-1})(1-z^{-1}q^{4n-1})} 
			+ 
			\sum_{n=1}^{\infty}
			\frac{(1-z)(1-z^{-1})q^{8n^2 + 2}(1 + q^{4n})}{(1-zq^{4n})(1 - z^{-1}q^{4n})} 
			\\& \quad
			- 
			\sum_{n =0}^{\infty}
			\frac{(1-z)(1-z^{-1})q^{8n^2 + 6n + 1}}{(1 -zq^{4n+1})(1 - z^{-1}q^{4n+1})}
		\Bigg)
	,
\end{align*}
and obtain
\begin{align*}
	&\sum_{n = 1}^{\infty}\eta_{2k}^{\X{42}}(n)q^n 
	= 
		\frac{1}{(q)_{\infty}} 
		\Bigg( 
			\sum_{n = 1}^{\infty} \frac{q^{8n^2 - 2n + (4n-1)k}}{(1-q^{4n-1})^{2k}}  
			- 
			\sum_{n = 1}^{\infty} \frac{q^{8n^2 - 2n + 4nk}(1 + q^{4n})}{(1-q^{4n})^{2k}} 
			+
			\sum_{n = 0}^{\infty} \frac{q^{8n^2 + 2n + (4n + 1)k}}{(1 - q^{4n + 1})^{2k}}
		\Bigg)
	,\\
	&\sum_{n = 1}^{\infty} \hspt{\X{42}}{k}{n}q^{n} 
	:= 
		\sum_{n = 1}^{\infty}\left(\mu_{2k}(n)-\eta_{2k}^{\X{42}}(n)\right)q^{n} 
	= 
		\sum_{n_{k} \geq \dotsb \geq n_{1} \geq 1} 
		\frac{(-q^2;q^2)_{n_{1}-1} q^{2n_{1} + n_{2} + \dotsb + n_{k}}}
		{(q^{n_1+1})_{n_1} (q^{n_{1} + 1})_{\infty}
			(1-q^{n_{k}})^{2}\dotsm (1-q^{n_{1}})^{2}}
	.
\end{align*}

(14) Using the Bailey pair E(4) from \cite{Slater1}, relative to $(1,q)$,
\begin{align*} 
	\beta_n 
	&= 
		\frac{q^n}{(q^2; q^2)_n} 
	,&
 	\alpha_n 
 	&=
		\begin{cases} 
		1 							& n=0 
		,\\ 
		(-1)^nq^{n^2-n}(1+q^{2n}) 	& n\geq 1 
		,\\
		\end{cases}
\end{align*}
we define
\begin{align*}
	R_{E4}(z,q)
	&=
		\frac{\aprod{-q}{\infty}}{\aprod{q}{\infty}} 
		\Bigg(
			1
			+ 
			\sum\limits_{n=1}^\infty 
			\frac{(1-z)(1-z^{-1})(-1)^nq^{n^2}(1+q^{2n})}{(1-zq^n)(1-z^{-1}q^n)} 
		\Bigg)
	,
\end{align*} 
and obtain
\begin{align*} 
	\sum_{n = 1}^{\infty}\eta_{2k}^{E4}(n)q^n 
	&=
		\frac{(-q)_\infty}{(q)_\infty}
		\sum\limits_{n=1}^\infty \frac{(-1)^{n+1}q^{n^2-n+kn}(1+q^{2n})}{(1-q^n)^{2k}}
	,\\
	\sum_{n = 1}^{\infty} \hspt{E4}{k}{n}q^{n} 
	&:= 
		\sum_{n = 1}^{\infty}\left(\overline{\mu}_{2k}(n) - \eta_{2k}^{E4}(n)\right)q^{n} 
    =  
    	\sum\limits_{n_k \geq \dotsb \geq n_1 \geq 1}
    	\frac{(-q^{n_1+1})_\infty q^{2n_1+ n_2+\dotsb +n_k}} 
    	{(q^{n_1+1})_\infty (1-q^{n_k})^2\dotsm (1-q^{n_1})^2}
	.
\end{align*} 

(15) Using the Bailey pair I(14) from \cite{Slater2}, relative to $(1,q)$,
\begin{align*}
	\beta_{n} 
	&= 
	 	\begin{cases} 
		1 																& n=0 
		,\\ 
		\frac{\aqprod{-q^{2}}{q^{2}}{n-1}}
		{\aqprod{q}{q^{2}}{n} \aprod{q}{n} \aprod{-q}{n-1}} 	& n \geq 1
		,
		\end{cases} 
	&
	\alpha_{n} 
	&=
		\begin{cases}
		1 									& n = 0 
		,\\
		(-1)^{k}q^{2k^2 - k}(1 + q^{2k}) 	& n = 2k 
		,\\
		0 									& n = 2k + 1
		,
		\end{cases}
\end{align*}
we define
\begin{align*}
	R_{I14}(z,q) 
	&= 
		\frac{\aprod{-q}{\infty}}{\aprod{q}{\infty}}
		\Bigg(
			1 
			+ 
			\sum_{n = 1}^{\infty}
			\frac{(1-z)(1-z^{-1})(-1)^{n}q^{2n^2 + n}(1 + q^{2n})}{(1-zq^{2n})(1-z^{-1}q^{2n})}
		\Bigg) 
	,
\end{align*}
and obtain
\begin{align*}
	\sum_{n = 1}^{\infty}\eta_{2k}^{I14}(n)q^n 
	&= 
		\frac{(-q)_{\infty}}{(q)_{\infty}}
		\sum_{n = 1}^{\infty} \frac{(-1)^{n+1}q^{2n^2 - n + 2nk}(1 + q^{2n})}{(1-q^{2n})^{2k}}
	,\\
	\sum_{n = 1}^{\infty} \hspt{I14}{k}{n}q^{n} 
	&:= 
		\sum_{n = 1}^{\infty}\left(\overline{\mu}_{2k}(n)-\eta_{2k}^{I14}(n)\right)q^{n} 
	= 
		\sum_{n_{k} \geq \dotsb \geq n_{1} \geq 1} 
		\frac{(-q^2;q^2)_{n_{1}-1}  (-q^{n_{1}})_{\infty} q^{n_{1} + \cdots{} + n_{k}}}
		{(q;q^{2})_{n_{1}} (q^{n_{1} + 1})_{\infty}
			(1-q^{n_{k}})^{2}\dotsm (1-q^{n_{1}})^{2}}
	.
\end{align*}

(16) Using the Bailey pair in Lemma 3.1 from \cite{BowmannMclaughlinSills1}, relative to $(1,q)$,
\begin{align*}
	\beta_{n} 
	&= 
		\begin{cases} 
		1 											& n=0 
		,\\ 
		\frac{\aqprod{-q^{3}}{q^{3}}{n-1}}
		{\aprod{-q}{n} \aprod{q}{2n-1}} 	& n \geq 1
		,
		\end{cases} 
	&
	\alpha_{n} =
		\begin{cases}
		1 											& n = 0 
		,\\
		(-1)^{k}q^{\frac{3k(3k-1)}{2}}(1 + q^{3k}) 	& n = 3k 
		,\\
		-2q^{18k^2 + 9k + 1} 						& n = 6k + 1 
		,\\
		2q^{18k^2 + 15k + 3} 						& n = 6k + 2 
		,\\
		2q^{18k^2 + 21k + 6} 						& n = 6k + 4 
		,\\
		-2q^{18k^2 + 27k + 10} 						& n = 6k + 5
		,
	\end{cases}
\end{align*}
we define
\begin{align*}
	R_{\X{46}}(z,q)
	&= 
		\frac{\aprod{-q}{\infty}}{\aprod{q}{\infty}} 
		\Bigg(
			1 
			+ 
			\sum_{n = 1}^{\infty}
			\frac{(1-z)(1-z^{-1})(-1)^{n}q^{\frac{3n(3n+1)}{2}}(1 + q^{3n})}{(1-zq^{3n})(1-z^{-1}q^{3n})} 
			\\&\quad			
			- 
			2\sum_{n = 0}^{\infty}
			\frac{(1-z)(1-z^{-1})q^{18n^2 + 15n + 2}}{(1 - zq^{6n+1})(1 - z^{-1}q^{6n+1})} 
			+ 
			2\sum_{n = 0}^{\infty}
			\frac{(1-z)(1-z^{-1})q^{18n^2 + 21n + 5}}{(1-zq^{6n+2})(1 - z^{-1}q^{6n+2})} 
			\\&\quad
			+ 
			2\sum_{n = 0}^{\infty}
			\frac{(1-z)(1-z^{-1})q^{18n^2 + 27n + 10}}{(1 -zq^{6n+4})(1 - z^{-1}q^{6n+4})} 
			- 
			2\sum_{n = 0}^{\infty}\frac{(1-z)(1-z^{-1})q^{18n^2 + 33n + 15}}{(1 -zq^{6n + 5})(1 - z^{-1}q^{6n+5})} 
		\Bigg)
	,
\end{align*}
and obtain
\begin{align*}
	&\sum_{n = 1}^{\infty}\eta_{2k}^{\X{46}}(n)q^n 
	= 
		\frac{(-q)_{\infty}}{(q)_{\infty}}
		\Bigg(
			\sum_{n = 1}^{\infty} \frac{(-1)^{n+1}q^{\frac{3n(3n-1)}{2} + 3nk}(1 + q^{3n})}{(1-q^{3n})^{2k}} 
			+ 
			2\sum_{n = 0}^{\infty} \frac{q^{18n^2 + 9n + 1 + (6n+1)k}}{(1-q^{6n + 1})^{2k}} 
			\\&\quad
			- 
			2\sum_{n = 0}^{\infty} \frac{q^{18n^2 + 15n + 3 + (6n + 2)k}}{(1 - q^{6n + 2})^{2k}} 
			- 
			2\sum_{n = 0}^{\infty} \frac{q^{18n^2 + 21n + 6 + (6n + 4)k}}{(1 - q^{6n + 4})^{2k}} 
			+ 
			2\sum_{n = 0}^{\infty} \frac{q^{18n^2 + 27n + 10 + (6n + 5)k}}{(1 - q^{6n + 5})^{2k}}
		\Bigg)
	,\\
	&\sum_{n = 1}^{\infty} \hspt{\X{46}}{k}{n}q^{n} 
	:= 
		\sum_{n = 1}^{\infty}\left(\overline{\mu}_{2k}(n)-\eta_{2k}^{\X{46}}(n)\right)q^{n} 
	= 
		\sum_{n_{k} \geq \dotsb \geq n_{1} \geq 1} 
		\frac{(-q^3;q^3)_{n_{1}-1} (-q^{n_{1} + 1})_{\infty} q^{n_1+ \dotsb + n_k}}
		{(q^{n_{1} + 1})_{n_{1}-1} (q^{n_1+1})_{\infty} (1-q^{n_{k}})^{2} \dotsm (1-q^{n_{1}})^{2}}
	.
\end{align*}

(17) Using the Bailey pair J(1) from \cite{Slater2}, 
which also appears as equation (3.8) in \cite{Slater1},
relative to $(1,q)$,
\begin{align*}
	\beta_{n} 
  	&= 
 		\begin{cases} 
		1 																		& n=0 
		,\\ 
		\frac{\aqprod{q^{3}}{q^{3}}{n-1}}{\aprod{q}{2n-1}\aprod{q}{n}}	& n \geq 1
		,
		\end{cases} 
	&
	\alpha_{n}
	&=
		\begin{cases}
		1 									& n = 0 
		,\\
		0 									& n  = 3k-1 
		,\\
		(-1)^{k}q^{\frac{3k(3k-1)}{2}}(1 + q^{3k}) 	& n = 3k 
		,\\
        0 & n = 3k+1
        ,
		\end{cases}
\end{align*}
we define
\begin{align*}
	R_{J1}(z,q)
	&=
		\frac{1}{\aprod{q}{\infty}\aqprod{q^3}{q^3}{\infty}}
		\Bigg(
			1 
			+ 
			\sum_{n = 1}^{\infty}
			\frac{(1-z)(1-z^{-1})(-1)^{n}q^{\frac{3n(3n+1)}{2}}(1+q^{3n})}{(1-zq^{3n})(1-z^{-1}q^{3n})}
		\Bigg) 
	,
\end{align*}
and obtain
\begin{align*} 
	\sum_{n = 1}^{\infty}\eta_{2k}^{J1}(n)q^n 
	&= 
		\frac{1}{(q)_{\infty}(q^3;q^3)_\infty}
		\sum_{n = 1}^{\infty} \frac{(-1)^{n+1}q^{\frac{3n(3n-1)}{2}+3nk}(1 + q^{3n})}{(1-q^{3n})^{2k}} 
	,\\
	\sum_{n = 1}^{\infty} \hspt{J1}{k}{n}q^{n} 
	&:= 
		\sum_{n = 1}^{\infty}\left(\mu^J_{2k}(n) - \eta_{2k}^{J1}(n)\right)q^{n} 
    \\
    &= 
    	\sum_{n_{k} \geq \dotsb \geq n_{1} \geq 1}
   		\frac{q^{n_{1} + \dotsb + n_{k}}}
   		{(q)_{2n_1-1} (q^{3n_1};q^3)_\infty (q^{n_{1} + 1})_{\infty} (1-q^{n_{k}})^{2}\dotsm (1-q^{n_{1}})^{2}}
	.
\end{align*}

(18) Using the Bailey pair J(2) from \cite{Slater2}, 
which also appears unlabeled on page 467 of \cite{Slater1},
relative to $(1,q)$,
\begin{align*}
	\beta_{n} 
	&= 
		\begin{cases} 
		1 										& n=0 
		,\\ 
		\frac{\aqprod{q^{3}}{q^{3}}{n-1}}
		{\aprod{q}{2n} \aprod{q}{n-1}} 	& n \geq 1
		,		
		\end{cases} 
	&
	\alpha_{n}
	&=
		\begin{cases}
		1 										& n = 0 
		,\\
		(-1)^{k+1}q^{\frac{9k(k-1)}{2}+1}  				& n  = 3k-1 
		,\\
		(-1)^{k}q^{\frac{3k(3k-1)}{2}}(1 + q^{3k}) 		& n = 3k 
		,\\
        (-1)^{k+1}q^{\frac{9k(k+1)}{2}+1}  				& n = 3k+1
		,
		\end{cases}
\end{align*}
we define
\begin{align*}
	R_{J2}(z,q)
	&= 
		\frac{1}{\aprod{q}{\infty} \aqprod{q^3}{q^3}{\infty}}
		\Bigg( 
			1
			+
			\sum\limits_{n=1}^\infty 
			\frac{(1-z)(1-z^{-1})(-1)^{n+1}q^{\frac{3n(3n-1)}{2}}}{(1-zq^{3n-1})(1-z^{-1}q^{3n-1})} 
			\\&\quad
			+
			\sum\limits_{n=0}^\infty 
			\frac{(1-z)(1-z^{-1})(-1)^{n+1}q^{\frac{3n(3n+5)}{2} + 2}}{(1-zq^{3n+1})(1-z^{-1}q^{3n+1})}   
			+
			\sum\limits_{n=1}^\infty 
			\frac{(1-z)(1-z^{-1})(-1)^{n}q^{\frac{3n(3n+1)}{2}}(1 + q^{3n})}{(1-zq^{3n})(1-z^{-1}q^{3n})} 
		\Bigg)
	,
\end{align*}
and obtain
\begin{align*}
	\sum_{n = 1}^{\infty}\eta_{2k}^{J2}(n)q^n 
	&=  
		\frac{1}{(q)_{\infty}(q^3;q^3)_\infty}
		\Bigg(
			\sum_{n=1}^\infty \frac{(-1)^{n}q^{ \frac{9n(n-1)}{2} +1+ (3n-1)k}}{(1-q^{3n-1})^{2k}} 
			+
			\sum_{n=0}^\infty \frac{(-1)^{n}q^{ \frac{9n(n+1)}{2} +1+(3n+1)k}}{(1-q^{3n+1})^{2k}}   
			\\&\quad
			-
			\sum_{n=1}^\infty \frac{(-1)^{n}q^{ \frac{3n(3n-1)}{2}+3nk}(1 + q^{3n})}{(1-q^{3n})^{2k}} 
		\Bigg) 
	,\\ 
	\sum_{n = 1}^{\infty} \hspt{J2}{k}{n}q^{n} 
	&:= 
		\sum_{n = 1}^{\infty}\left(\mu^J_{2k}(n) - \eta_{2k}^{J2}(n)\right)q^{n} 
	\\
    &= 
    	\sum_{n_{k} \geq \dotsb \geq n_{1} \geq 1}
    	\frac{q^{n_{1} + \dotsb + n_{k}}}
    	{(q)_{n_{1}-1}  (q^{n_{1}+1})_{n_{1}}  (q^{3n_1};q^3)_\infty (q^{n_{1} + 1})_{\infty}
    		(1-q^{n_{k}})^{2}\dotsm (1-q^{n_{1}})^{2}}
	.
\end{align*}
    
(19) Using the Bailey pair J(3) from \cite{Slater2}, 
which also appears unlabeled on page 467 of \cite{Slater1},
relative to $(1,q)$,
\begin{align*}
	\beta_{n} 
	&= 
		\begin{cases} 
		1 										& n=0 
		,\\ 
		\frac{q^n\aqprod{q^{3}}{q^{3}}{n-1}}
		{\aprod{q}{2n} \aprod{q}{n-1}} 	& n \geq 1
		,
		\end{cases} 
	&
	\alpha_{n}
	&=
		\begin{cases}
		1 									& n = 0 
		,\\
		(-1)^{k+1}q^{\frac{3k(3k-1)}{2}}  			& n  = 3k-1 
		,\\
		(-1)^{k}q^{\frac{3k(3k-1)}{2}}(1 + q^{3k}) 	& n = 3k 
		,\\
        (-1)^{k+1}q^{\frac{3k(3k+1)}{2}}  			& n = 3k+1
		,
		\end{cases}
\end{align*}
we define
\begin{align*}
	R_{J3}(z,q)
	&= 
		\frac{1}{\aprod{q}{\infty} \aqprod{q^3}{q^3}{\infty}}
		\Bigg(
			1
			+
			\sum\limits_{n=1}^\infty 
			\frac{(1-z)(1-z^{-1})(-1)^{n+1}q^{\frac{3n(3n+1)}{2}-1}}{(1-zq^{3n-1})(1-z^{-1}q^{3n-1})} 
			\\&\quad
			+
			\sum\limits_{n=0}^\infty 
			\frac{(1-z)(1-z^{-1})(-1)^{n+1}q^{\frac{9n(n+1)}{2}+1}}{(1-zq^{3n+1})(1-z^{-1}q^{3n+1})}   
			+
			\sum\limits_{n=1}^\infty 
			\frac{(1-z)(1-z^{-1})(-1)^{n}q^{\frac{3n(3n+1)}{2}}(1 + q^{3n})}{(1-zq^{3n})(1-z^{-1}q^{3n})} 
		\Bigg)  
	,
\end{align*}
and obtain
\begin{align*}
	\sum_{n = 1}^{\infty}\eta_{2k}^{J3}(n)q^n 
	&=  
		\frac{1}{(q)_{\infty}(q^3;q^3)_\infty}
		\Bigg(
			\sum\limits_{n=1}^\infty \frac{(-1)^{n} q^{ \frac{3n(3n-1)}{2} +(3n-1)k}}{(1-q^{3n-1})^{2k}} 
			+
			\sum\limits_{n=0}^\infty \frac{(-1)^{n}q^{ \frac{3n(3n+1)}{2} +(3n+1)k}}{(1-q^{3n+1})^{2k}}   
			\\&\quad
			-
			\sum\limits_{n=1}^\infty \frac{(-1)^{n}q^{ \frac{3n(3n-1)}{2} +3nk}(1 + q^{3n})}{(1-q^{3n})^{2k}} 
		\Bigg)  
	,\\
	\sum_{n = 1}^{\infty} \hspt{J3}{k}{n}q^{n} 
	&:= 
		\sum_{n = 1}^{\infty}\left(\mu^J_{2k}(n) - \eta_{2k}^{J3}(n)\right) q^{n} 
	\\
    &= 
    	\sum_{n_{k} \geq \dotsb \geq n_{1} \geq 1}
    	\frac{q^{2n_{1} + n_{2} + \dotsb + n_{k}}}
    	{(q)_{n_{1}-1} (q^{n_{1}+1})_{n_{1}} (q^{3n_1};q^3)_\infty (q^{n_{1} + 1})_{\infty}
    		(1-q^{n_{k}})^{2}\dotsm (1-q^{n_{1}})^{2}}
	.
\end{align*}
    
(20) Using the Bailey pair in the ninth entry in the table on page 470 of \cite{Slater1},
relative to $(1,q)$,
\begin{align*}
	\beta_{n} 
	&= 
		\begin{cases} 
		1 																& n=0 
		,\\ 
		\frac{\aqprod{q^{2}}{q^{2}}{n-1}}
		{\aqprod{q}{q^{2}}{n} \aprod{q}{n} \aprod{q}{n-1}} 		& n \geq 1
		,
		\end{cases} 
	&
	\alpha_{n}
	&=
		\begin{cases}
		1 							& n = 0 
		,\\
		q^{2k^2 - k}(1 + q^{2k}) 	& n = 2k 
		,\\
		0 							& n = 2k + 1
		,
		\end{cases}
\end{align*}
we define
\begin{align*}
	R_{\X{40}}(z,q) 
	&= 
		\frac{1}{\aprod{q}{\infty} \aqprod{q^2}{q^2}{\infty}}
		\Bigg(
			1 
			+ 
			\sum_{n = 1}^{\infty}
			\frac{(1-z)(1-z^{-1})q^{2n^2+n}(1+q^{2n})}{(1-zq^{2n})(1-z^{-1}q^{2n})}
		\Bigg)
	,
\end{align*}
and obtain
\begin{align*}
	\sum_{n = 1}^{\infty}\eta_{2k}^{\X{40}}(n)q^n 
	&= 
		\frac{-1}{(q)_{\infty}(q^2;q^2)_{\infty}}
		\sum_{n = 1}^{\infty} \frac{q^{2n^2 - n + 2nk}(1 + q^{2n})}{(1-q^{2n})^{2k}}
	,\\
	\sum_{n = 1}^{\infty} \hspt{\X{40}}{k}{n}q^{n} 
	&:= 
		\sum_{n = 1}^{\infty}\left(\mu^{\X{40}}_{2k}(n)-\eta_{2k}^{\X{40}}(n)\right)q^{n} 
	\\
	&= 
		\sum_{n_{k} \geq \dotsb \geq n_{1} \geq 1} 
		\frac{q^{n_{1} + \dotsb + n_{k}}}
		{(q)_{n_1 - 1} (q;q^2)_{n_1} (q^{2n_1};q^2)_{\infty}  (q^{n_1 + 1})_{\infty}
			(1-q^{n_{k}})^{2}\dotsm (1-q^{n_{1}})^{2}}
	.
\end{align*}

(21) Using the Bailey pair F(3) from \cite{Slater1}, relative to $(1,q^2)$,
\begin{align*}
	\beta_n 
	&= 
		\frac{1}{q^n(q)_{2n}} 
	,&
	\alpha_n 
	&=
		\begin{cases} 
		1 				& n=0 
		,\\ 
		q^{n}+q^{-n} 	& n\geq 1 
		,
	\end{cases}
\end{align*}
we define
\begin{align*}
	R_{F3}(z,q)
	&= 
		\frac{1}{\aqprod{q^2}{q^2}{\infty}}
		\Bigg(
			1
			+
			\sum\limits_{n=1}^\infty
			\frac{(1-z)(1-z^{-1})q^{n}(1+q^{2n})}{(1-zq^{2n})(1-z^{-1}q^{2n})}
		\Bigg) 
	,
\end{align*}
and obtain
\begin{align*}
	&
	\sum\limits_{n=1}^\infty \eta^{F3}_{2k}(n)q^n 
	=
		\frac{-1}{(q^2;q^2)_\infty}
		\sum_{n=1}^\infty \frac{q^{2nk-n}(1+q^{2n})}{(1-q^{2n})^{2k}} 
	,\\
	&
	\sum_{n = 1}^{\infty} \hspt{F3}{k}{n}q^{n} 
	:= 
		\sum_{n = 1}^{\infty}\left(\mu^{F}_{2k}(n) - \eta_{2k}^{F3}(n)\right)q^{n} 
	=
		\sum\limits_{n_k \geq \dotsb \geq n_1 \geq 1} 
		\frac{q^{n_1+2n_2\dotsb +2n_k}} 
		{(q;q^2)_{n_1} (q^{2n_1+2};q^2)_{\infty}(1-q^{2n_k})^2 \dotsm (1-q^{2n_1})^2} 
	.
\end{align*}

(22) Using the Bailey pair G(1) from \cite{Slater1}, 
which is also L(3) from \cite{Slater2},
relative to $(1,q^2)$,
\begin{align*} 
	\beta_n 
	&= 
		\frac{1}{\aqprod{-q}{q^2}{n} \aqprod{q^4}{q^4}{n}} 
	,&
 	\alpha_n 
 	&=
		\begin{cases} 
		1 								& n=0
		,\\ 
		(-1)^nq^{\frac{n(3n-1)}{2}}(1+q^{n}) 	& n\geq 1
		,
		\end{cases}
\end{align*}
we define
\begin{align*}
	R_{G1}(z,q)
	&=
		\frac{\aqprod{-q}{q^2}{\infty}}{\aqprod{q^2}{q^2}{\infty}^2}		
		\Bigg(
			1	 
			+ 
			\sum_{n = 1}^{\infty}
			\frac{(1-z)(1-z^{-1})(-1)^{n}q^{\frac{3n(n+1)}{2}}(1+q^{n})}{(1-zq^{2n})(1-z^{-1}q^{2n})}
		\Bigg) 
	,
\end{align*}
and obtain
\begin{align*}
	\sum_{n = 1}^{\infty}\eta_{2k}^{G1}(n)q^n 
	&= 
		\frac{(-q;q^2)_{\infty}}{(q^2;q^2)^2_{\infty}}
		\sum_{n = 1}^{\infty} \frac{(-1)^{n+1}q^{ \frac{n(3n-1)}{2} +2nk}(1+q^{n})}{(1-q^{2n})^{2k}} 
	,\\
	\sum_{n = 1}^{\infty} \hspt{G1}{k}{n}q^{n} 
	&:= 
		\sum_{n = 1}^{\infty}\left(\mu^G_{2k}(n)- \eta_{2k}^{G1}(n)\right)q^{n} 
    =  
    	\sum\limits_{n_k \geq \dotsb \geq n_1 \geq 1} 
    	\frac{(-q^{2n_1+1};q^2)_\infty q^{2n_1+\dotsb +2n_k}} 
    	{(q^4;q^4)_{n_1} (q^{2n_1+2};q^2)^2_{\infty}
    		(1-q^{2n_k})^2\dotsm (1-q^{2n_1})^2}
.
\end{align*}

(23) Using the Bailey pair G(3) from \cite{Slater1}, relative to $(1,q^2)$,
\begin{align*} 
	\beta_n 
	&= 
		\frac{q^{2n}}{\aqprod{-q}{q^2}{n} \aqprod{q^4}{q^4}{n}} 
	,&
 	\alpha_n 
 	&=
		\begin{cases} 
		1 								& n=0 
		,\\ 
		(-1)^nq^{\frac{3n(n-1)}{2}}(1+q^{3n}) 	& n\geq 1 
		,
		\end{cases}
\end{align*}
we define
\begin{align*}
	R_{G3}(z,q)
	&=
		\frac{\aqprod{-q}{q^2}{\infty}}{\aqprod{q^2}{q^2}{\infty}^2}
		\Bigg(
			1 
			+ 
			\sum_{n = 1}^{\infty}
			\frac{(1-z)(1-z^{-1})(-1)^{n}q^{\frac{n(3n+1)}{2}}(1+q^{3n})}{(1-zq^{2n})(1-z^{-1}q^{2n})}
		\Bigg) 
	,
\end{align*}
and obtain
\begin{align*}
	\sum_{n = 1}^{\infty}\eta_{2k}^{G3}(n)q^n 
	&= 
		\frac{(-q;q^2)_{\infty}}{(q^2;q^2)^2_{\infty}}
		\sum_{n = 1}^{\infty} \frac{(-1)^{n+1}q^{ \frac{3n(n-1)}{2} +2nk}(1+q^{3n})}{(1-q^{2n})^{2k}} 
	,\\
	\sum_{n = 1}^{\infty} \hspt{G3}{k}{n}q^{n} 
	&:= 
		\sum_{n = 1}^{\infty}\left(\mu^G_{2k}(n)-\eta_{2k}^{G3}(n)\right)q^{n} 
    =  
    	\sum\limits_{n_k \geq \dotsb \geq n_1 \geq 1} 
    	\frac{(-q^{2n_1+1};q^2)_\infty q^{4n_1+2n_2+\dotsb+2n_k}} 
    	{(q^4;q^4)_{n_1} (q^{2n_1+2};q^2)^2_{\infty}
    		(1-q^{2n_k})^2\dotsm (1-q^{2n_1})^2}
	.
\end{align*}

(24) Using the Bailey pair in the first entry in the table on page 470 of \cite{Slater1}, 
relative to $(1,q^2)$,
\begin{align*}
	\beta_{n} 
	&= 
		\frac{q^{n^2 - 2n}}{\aqprod{q^4}{q^4}{n} \aqprod{q}{q^2}{n}} 
	,&
	\alpha_{n} 
	&=
		\begin{cases} 
		1 											& n = 0 
		,\\	
		(-1)^{k}q^{2k^2 - 3k}(1 + q^{6k}) 			& n = 2k 
		,\\
		(-1)^{k}q^{2k^2 - k - 1}(1 - q^{6k + 3}) 	& n = 2k + 1
		,
		\end{cases}
\end{align*}
we define
\begin{align*}
	R_{Y1}(z,q) 
	&= 
		\frac{1}{\aqprod{q^2}{q^2}{\infty}^2}
		\Bigg(
			1 
			+ 
			\sum_{n =1}^{\infty}
			\frac{(1-z)(1-z^{-1})(-1)^{n}q^{2n^2 + n}(1 + q^{6n})}{(1-zq^{4n})(1-z^{-1}q^{4n})} 
			\\&\quad 
			+
			\sum_{n = 0}^{\infty}
			\frac{(1-z)(1-z^{-1})(-1)^{n}q^{2n^2+3n+1}(1 - q^{6n + 3})}{(1-zq^{4n+2})(1-z^{-1}q^{4n+2})}
		\Bigg)
	,
\end{align*}
and obtain
\begin{align*}
	\sum_{n = 1}^{\infty}\eta_{2k}^{Y1}(n)q^n 
	&= 
		\frac{1}{(q^2;q^2)_{\infty}^2} 
		\left(
			\sum_{n = 1}^{\infty}\frac{(-1)^{n+1}q^{2n^2 - 3n + 4nk}(1 + q^{6n})}{(1-q^{4n})^{2k}} 
			\right.\\&\quad\left.
			+ 
			\sum_{n = 0}^{\infty}\frac{(-1)^{n+1}q^{2n^2 - n - 1 + (4n+2)k}(1 - q^{6n + 3})}{(1-q^{4n+2})^{2k}}
		\right)
	,\\
	\sum_{n = 1}^{\infty} \hspt{Y1}{k}{n}q^{n} 
	&:= 
		\sum_{n = 1}^{\infty}\left(\mu^{Y}_{2k}(n) - \eta_{2k}^{Y1}(n)\right)q^{n} 
	\\&= 
		\sum_{n_{k} \geq \dotsb \geq n_{1} \geq 1}
		\frac{q^{n_{1}^2 + 2n_{2} + \dotsb + 2n_{k}}}
		{(q;q^2)_{n_{1}} (q^4;q^4)_{n_{1}}  (q^{2n_{1} + 2} ; q^2)_{\infty}^{2} 
			(1-q^{2n_{k}})^{2} \dotsm (1-q^{2n_{1}})^{2}}
	.
\end{align*}

(25) Using the Bailey pair in the second entry in the table on page 470 of \cite{Slater1}, 
relative to $(1,q^2)$,
\begin{align*}
	\beta_{n} 
	&= 
		\frac{q^{n^2}}{\aqprod{q^4}{q^4}{n} \aqprod{q}{q^2}{n}} 
	,&
	\alpha_{n} 
	&=
		\begin{cases} 
		1 										& n = 0 
		,\\	
		(-1)^{k}q^{2k^2 - k}(1 + q^{2k}) 		& n = 2k 
		,\\
		(-1)^{k+1}q^{2k^2 + k}(1 - q^{2k+1}) 	& n = 2k + 1
		,
		\end{cases}
\end{align*}
we define
\begin{align*}
	R_{Y2}(z,q) 
	&= 
		\frac{1}{\aqprod{q^2}{q^2}{\infty}^2}
		\Bigg(
			1 
			+ 
			\sum_{n =1}^{\infty}
			\frac{(1-z)(1-z^{-1})(-1)^{n}q^{2n^2 + 3n}(1 + q^{2n})}{(1-zq^{4n})(1-z^{-1}q^{4n})} 
			\\& \quad
			+
			\sum_{n = 0}^{\infty}
			\frac{(1-z)(1-z^{-1})(-1)^{n+1}q^{2n^2 +5n+2}(1 - q^{2n + 1})}{(1-zq^{4n+2})(1-z^{-1}q^{4n+2})}
		\Bigg)
	,
\end{align*}
and obtain
\begin{align*}
	\sum_{n = 1}^{\infty}\eta_{2k}^{Y2}(n)q^n 
	&= 
		\frac{1}{(q^2;q^2)_{\infty}^{2}}
		\left(
			\sum_{n = 1}^{\infty}\frac{(-1)^{n+1}q^{2n^2 - n + 4nk}(1 + q^{2n})}{(1-q^{4n})^{2k}}
			+
		 	\sum_{n = 0}^{\infty}\frac{(-1)^{n}q^{2n^2 + n + (4n+2)k}(1 - q^{2n+1})}{(1-q^{4n+2})^{2k}}
		 \right)
	,\\
	\sum_{n = 1}^{\infty} \hspt{Y2}{k}{n}q^{n} 
	&:= 
		\sum_{n = 1}^{\infty}\left(\mu^{Y}_{2k}(n)- \eta_{2k}^{Y2}(n)\right)q^{n} 
	\\
	&= 
		\sum_{n_{k} \geq \dotsb \geq n_{1} \geq 1}
		\frac{q^{n_{1}^2 + 2n_{1} + 2n_{2} + \dotsb + 2n_{k}}}
		{(q;q^2)_{n_{1}} (q^4;q^4)_{n_{1}} (q^{2n_{1} + 2} ; q^2)_{\infty}^{2}
			(1-q^{2n_{k}})^{2}\dotsm (1-q^{2n_{1}})^{2}}
	.
\end{align*}

(26) Using the Bailey pair in the third entry in the table on page 470 of \cite{Slater1}, 
relative to $(1,q^2)$,
\begin{align*}
	\beta_{n} 
	&= 
		\frac{1}{\aqprod{q^4}{q^4}{n} \aqprod{q}{q^2}{n}} 
	,&
	\alpha_{n} 
	&=
		\begin{cases} 
		1 											& n = 0 
		,\\	
		(-1)^{k}q^{6k^2 - k}(1 + q^{2k}) 			& n = 2k 
		,\\
		(-1)^{k}q^{6k^2 + 5k + 1}(1 - q^{2k+1}) 	& n = 2k + 1
		,
		\end{cases}
\end{align*}
we define
\begin{align*}
	R_{Y3}(z,q) 
	&= 
		\frac{1}{\aqprod{q^2}{q^2}{\infty}^2}
		\Bigg(
			1 
			+ 
			\sum_{n =1}^{\infty}
			\frac{(1-z)(1-z^{-1})(-1)^{n}q^{6n^2 + 3n}(1 + q^{2n})}{(1-zq^{4n})(1-z^{-1}q^{4n})} 
			\\&\quad
			+
			\sum_{n = 0}^{\infty}
			\frac{(1-z)(1-z^{-1})(-1)^{n}q^{6n^2 + 9n+3}(1 - q^{2n + 1})}{(1-zq^{4n+2})(1-z^{-1}q^{4n+2})}
		\Bigg)
	,
\end{align*}
and obtain
\begin{align*}
	\sum_{n = 1}^{\infty}\eta_{2k}^{Y3}(n)q^n 
	&= 
		\frac{1}{(q^2;q^2)_{\infty}^2} 
		\left(
			\sum_{n = 1}^{\infty}\frac{(-1)^{n+1}q^{6n^2 - n + 4nk}(1 + q^{2n})}{(1-q^{4n})^{2k}} 
			\right. \\&\quad \left.
			+
			\sum_{n = 0}^{\infty}\frac{(-1)^{n+1}q^{6n^2 + 5n + 1 + (4n + 2)k}(1 - q^{2n+1})}{(1-q^{4n+2})^{2k}}
		\right)
	,\\
	\sum_{n = 1}^{\infty} \hspt{Y3}{k}{n}q^{n} 
	&:= 
		\sum_{n = 1}^{\infty}\left(\mu^{Y}_{2k}(n)-\eta_{2k}^{Y3}(n)\right)q^{n} 
	\\
	&= 
		\sum_{n_{k} \geq \dotsb \geq n_{1} \geq 1}
		\frac{q^{2n_1 + \dotsb + 2n_k}}
		{(q;q^2)_{n_{1}} (q^4;q^4)_{n_{1}} (q^{2n_{1} + 2} ; q^2)_{\infty}^{2} 
			(1-q^{2n_{k}})^{2} \dotsm (1-q^{2n_{1}})^{2}}
	.
\end{align*}

(27) Using the Bailey pair in the fourth entry in the table on page 470 of \cite{Slater1}, 
relative to $(1,q^2)$,
\begin{align*}
	\beta_{n} 
	&= 
		\frac{q^{2n}}{\aqprod{q^4}{q^4}{n} \aqprod{q}{q^2}{n}} 
	,&
	\alpha_{n} 
	&=
		\begin{cases} 
		1 										& n = 0 
		,\\	
		(-1)^{k}q^{6k^2 - 3k}(1 + q^{6k}) 		& n = 2k 
		,\\
		(-1)^{k+1}q^{6k^2 + 3k}(1 - q^{6k+3}) 	& n = 2k + 1
		,
		\end{cases}
\end{align*}
we define
\begin{align*}
	R_{Y4}(z,q) 
	&= 
		\frac{1}{\aqprod{q^2}{q^2}{\infty}^2}
		\Bigg(
			1 
			+ 
			\sum_{n =1}^{\infty}
			\frac{(1-z)(1-z^{-1})(-1)^{n}q^{6n^2 + n}(1 + q^{6n})}{(1-zq^{4n})(1-z^{-1}q^{4n})} 
			\\&\quad 
			+
			\sum_{n = 0}^{\infty}
			\frac{(1-z)(1-z^{-1})(-1)^{n+1}q^{6n^2 +7n+2}(1 - q^{6n + 3})}{(1-zq^{4n+2})(1-z^{-1}q^{4n+2})}
		\Bigg)
	,
\end{align*}
and obtain
\begin{align*}
	\sum_{n = 1}^{\infty}\eta_{2k}^{Y4}(n)q^n 
	&= 
		\frac{1}{(q^2;q^2)_{\infty}^2} 
		\left(
			\sum_{n = 1}^{\infty}\frac{(-1)^{n+1}q^{6n^2 - 3n + 4nk}(1 + q^{6n})}{(1-q^{4n})^{2k}} 
			\right.\\&\quad \left.
			+ 
			\sum_{n = 0}^{\infty}\frac{(-1)^{n}q^{6n^2 + 3m + (4n + 2)k}(1 - q^{6n+3})}{(1-q^{4n+2})^{2k}}
		\right)
	,\\
	\sum_{n = 1}^{\infty} \hspt{Y4}{k}{n}q^{n} 
  	&:= 
  		\sum_{n = 1}^{\infty}\left(\mu^{Y}_{2k}(n)-\eta_{2k}^{Y4}(n)\right)q^{n} 
  	\\
	&= 
		\sum_{n_{k} \geq \dotsb \geq n_{1} \geq 1}
		\frac{q^{4n_{1} + 2n_{2} + \dotsb + 2n_{k}}}
		{(q;q^2)_{n_{1}} (q^4;q^4)_{n_{1}} (q^{2n_{1} + 2} ; q^2)_{\infty}^{2}
			(1-q^{2n_{k}})^{2}\dotsm (1-q^{2n_{1}})^{2}}
	.
\end{align*}

(28) Using the Bailey pair L(2) from \cite{Slater2}, 
which is also M(1) from \cite{Slater2},
relative to $(1,q^4)$,
\begin{align*} 
	\beta_n 
	&= 
		\frac{\aqprod{q}{q^2}{2n}}{\aqprod{q^4}{q^4}{2n}} 
	,&
 	\alpha_n 
 	&=
		\begin{cases} 
		1 							& n=0 
		,\\ 
		(-1)^nq^{2n^2-n}(1+q^{2n}) 	& n\geq 1
		,
		\end{cases}
\end{align*}
we define
\begin{align*}
	R_{L2}(z,q)
	&= 
		\frac{\aqprod{-q}{q}\infty}{\aqprod{q^4}{q^4}{\infty}}
		\Bigg(
			1
			+
			\sum\limits_{n=1}^\infty\frac{(1-z)(1-z^{-1})(-1)^nq^{2n^2+3n}(1+q^{2n})}{(1-zq^{4n})(1-z^{-1}q^{4n})}
		\Bigg)
	,
\end{align*}
and obtain
\begin{align*}
	\sum\limits_{n=1}^\infty \eta^{L2}_{2k}(n)q^n 
	&=
		\frac{(-q)_\infty}{(q^4;q^4)_\infty}
		\sum_{n=1}^\infty 
		\frac{(-1)^{n+1}q^{2n^2-n+4nk}(1+q^{2n})}{(1-q^{4n})^{2k}} 
	,\\
	\sum_{n = 1}^{\infty} \hspt{L2}{k}{n}q^{n} 
	&:= 
		\sum_{n=1}^\infty \left(\mu^{L2}_{2k}(n)-\eta^{L2}_{2k}(n)\right)q^n 
	\\
	&=  
		\sum_{n_k \geq \dotsb \geq n_1 \geq 1} 
		\frac{q^{4n_1+\dotsb +4n_k}} 
		{(q^{4n_1+4};q^4)_{n_1} (q^{4n_1+1};q^2)_\infty (q^{4n_1+4};q^4)_{\infty}
			(1-q^{4n_k})^2 \dotsm (1-q^{4n_1})^2} 
	.
\end{align*}
\end{corollary}

As in our example, we see that in each case we have 
$\mu^{X}_{2k}(n)\ge\eta^{X}_{2k}(n)$. 
We note that the symmetrized moments 
$\eta^{X}_{2k}(n)$ for $X=F3$, $L5$, and $\X{40}$ are non-positive integers;
$\hspt{\X{38}}{k}{n}=2\hspt{\X{41}}{k}{n}$;
and $\hspt{\X{39}}{k}{n}=2\hspt{\X{42}}{k}{n}$.
By determining when 
$\mu^{X}_{2}(n)>\eta^{X}_{2}(n)$, we also obtain the strict inequalities for
$M^{X}_{2k}(n)>N^{X}_{2k}(n)$, upon noting that
\begin{align*} 
	M^X_{2k}(n)-N^{X}_{2k}(n) 
	&=
		\sum_{j=1}^{k}(2j)!S^*(k,j)( \mu^X_{2j}(n)-\eta^{X}_{2j}(n))
	\ge 
		\mu^X_{2}(n)-\eta^{X}_{2}(n)
.
\end{align*}
We record these inequalities in the following table.

\begin{center}
T{\sc able} 1. Strict Inequalities for $M^X_{2k}(n)>N^{X}_{2k}(n)$, for positive $k$.
\end{center}
\begin{align*}
\begin{array}{c|c||c|c||c|c}
	X & n & X & n & X & n	
	\\\hline
	A1&			n \geq 1&	
	\X{39}&		n \geq 2&
	F3&  		n \geq 1
	\\
	A3& 		n \geq 2&	
	\X{41}&		n \geq 1&
	G1&			n=2,{ } n \geq 4	
	\\
	A5& 		n \geq 2&	
	\X{42}&		n \geq 1&
	G3&			n=4,{ } n \geq 6
	\\
	A7&			n \geq 1&	
	E4&			n \geq 2&			
	Y1&			n \geq 1
	\\
	B2&			n \geq 2&	
	I14&		n \geq 2&
	Y2&			n \geq 3	
	\\	
	C1&			n \geq 1&			
	\X{46}& 	n \geq 1&
	Y3&			n \geq 2
	\\	
	C2&			n \geq 2&			
	J1&			n \geq 1&	
	Y4&			n \geq 4
	\\	
	C5& 		n \geq 1&			
	J2&			n \geq 1&	
	L2&			n=4,8,9, n\geq 11	 
	\\	
	L5&			n \geq 1& 	
	J3&			n \geq 2&	
	\\	
	\X{38}&		n \geq 1&	
	\X{40}&		n \geq 1&	
	\end{array}		
\end{align*}

We can actually determine inequalities between some of the ranks that compare
against the same crank. In 
particular, we see we have additional inequalities
for two rank moments $\eta^{X}_{2k}(n)$ and $\eta^{X^\prime}_{2k}(n)$ 
that are compared against the same crank when 
\begin{align*}
	\frac{P_{X}(q) \aprod{q}{n}^2\beta^{X}_n}{(1-q^n)^2}
	-	
	\frac{P_{X^\prime}(q) \aprod{q}{n}^2\beta^{X^\prime}_n}{(1-q^n)^2}
\end{align*}
clearly has non-negative coefficients. We record the identities that yield such
results in the follow corollary and omit the identities that would lead to an 
inequality that is already present. While we could also use these identities 
to deduce strict inequalities between the relevant ordinary moments, we leave 
that as an exercise to the interested reader.

\begin{corollary}
For positive $k$, we have that
\begin{align*}
	\sum_{n=1}^\infty \left(\eta^{B2}_{2k}(n)-\eta_{2k}(n)\right)q^n
	&=
		 \sum\limits_{n_k \geq \dotsc \geq n_1 \geq 1} 
		 \frac{q^{n_1+\dotsb+n_k}   } 
		 {(q^{n_1})_{\infty}(1-q^{n_k})^{2}\dotsm(1-q^{n_2})^2 }
	,\\
	\sum_{n=1}^\infty \left(\eta^{B2}_{2k}(n)-\eta^{A3}_{2k}(n)\right)q^n
	&=
		 \sum\limits_{n_k \geq \dotsc \geq n_1 \geq 1} 
		 \frac{q^{2n_1+n_2+\dotsb+n_k}   } 
		 {(q^{n_1+1})_{\infty}(1-q^{n_k})^{2}\dotsm(1-q^{n_1})^2 }
		 \times
		 \left(
		 	\frac{1}{\aprod{q^{n_1+1}}{n_1}}-1
		 \right)
	,\\
	\sum_{n=1}^\infty \left(\eta^{B2}_{2k}(n)-\eta^{C2}_{2k}(n)\right)q^n
	&=
		 \sum\limits_{n_k \geq \dotsc \geq n_1 \geq 1} 
		 \frac{q^{2n_1+n_2+\dotsb+n_k}   } 
		 {(q^{n_1+1})_{\infty}(1-q^{n_k})^{2}\dotsm(1-q^{n_1})^2 }
		 \times
		 \left(
		 	\frac{1}{\aqprod{q}{q^2}{n_1}}-1
		 \right)
	,\\
	\sum_{n=1}^\infty \left(\eta_{2k}(n)-\eta^{A1}_{2k}(n)\right)q^n
	&=
		 \sum\limits_{n_k \geq \dotsc \geq n_1 \geq 1} 
		 \frac{q^{n_1+\dotsb+n_k}   } 
		 {(q^{n_1+1})_{\infty}(1-q^{n_k})^{2}\dotsm(1-q^{n_1})^2 }
		 \times
		 \left(
		 	\frac{1}{\aprod{q^{n_1+1}}{n_1}}-1
		 \right)
	,\\
	\sum_{n=1}^\infty \left(\eta_{2k}(n)-\eta^{C1}_{2k}(n)\right)q^n
	&=
		 \sum\limits_{n_k \geq \dotsc \geq n_1 \geq 1} 
		 \frac{q^{n_1+\dotsb+n_k}   } 
		 {(q^{n_1+1})_{\infty}(1-q^{n_k})^{2}\dotsm(1-q^{n_1})^2 }
		 \times
		 \left(
		 	\frac{1}{\aqprod{q}{q^2}{n_1}}-1
		 \right)
	,\\
	\sum_{n=1}^\infty \left(\eta^{A5}_{2k}(n)-\eta^{A3}_{2k}(n)\right)q^n
	&=
		 \sum\limits_{n_k \geq \dotsc \geq n_1 \geq 1} 
		 \frac{q^{n_1+\dotsb+n_k}} 
		 {(q^{n_1+1})_{n_1}(q^{n_1})_{\infty}(1-q^{n_k})^{2}\dotsm(1-q^{n_2})^2}
		 \sum_{j=0}^{n_1-2}q^{jn_1}
	,\\
	\sum_{n=1}^\infty \left(\eta^{A5}_{2k}(n)-\eta^{A7}_{2k}(n)\right)q^n
	&=
		 \sum\limits_{n_k \geq \dotsc \geq n_1 \geq 1} 
		 \frac{q^{n_1^2+n_2+\dotsb+n_k}   } 
		 {(q^{n_1+1})_{n_1}(q^{n_1})_{\infty}(1-q^{n_k})^{2}\dotsm(1-q^{n_2})^2}
	,\\
	\sum_{n=1}^\infty \left(\eta^{A3}_{2k}(n)-\eta^{A1}_{2k}(n)\right)q^n
	&=
		 \sum\limits_{n_k \geq \dotsc \geq n_1 \geq 1} 
		 \frac{q^{n_1+\dotsb +n_k}} 
		 {(q^{n_1+1})_{n_1}(q^{n_1})_{\infty}(1-q^{n_k})^{2}\dotsm(1-q^{n_2})^2 }
	,\\
	\sum_{n = 1}^{\infty}\left(\eta_{2k}^{A3}(n)-\eta_{2k}^{\X{42}}(n)\right)q^{n} 
	&= 
		\sum_{n_{k} \geq \dotsb \geq n_{1} \geq 1} 
		\frac{q^{2n_{1}+n_2 + \dotsb + n_{k}}\left( (-q^2;q^2)_{n_{1}-1}-1 \right)   }
		{(q^{n_{1}+1})_{n_{1}} (q^{n_{1}+1})_{\infty}
			(1-q^{n_{k}})^{2}\dotsm (1-q^{n_{1}})^2}
	,\\
	\sum_{n=1}^\infty \left(\eta^{A7}_{2k}(n)-\eta^{A1}_{2k}(n)\right)q^n
	&=
		 \sum\limits_{n_k \geq \dotsc \geq n_1 \geq 1} 
		 \frac{q^{n_1+\dotsb +n_k}} 
		 {(q^{n_1+1})_{n_1}(q^{n_1})_{\infty}(1-q^{n_k})^{2}\dotsm(1-q^{n_2})^2 }
		 \sum_{j=0}^{n_1-2}q^{jn_1}
	,\\
	\sum_{n = 1}^{\infty}\left(\eta_{2k}^{A1}(n)-\eta_{2k}^{\X{41}}(n)\right)q^{n} 
	&= 
		\sum_{n_{k} \geq \dotsb \geq n_{1} \geq 1} 
		\frac{q^{n_{1} + \dotsb + n_{k}}\left( (-q^2;q^2)_{n_{1}-1}-1 \right)   }
		{(q^{n_{1}+1})_{n_{1}} (q^{n_{1}+1})_{\infty}
			(1-q^{n_{k}})^{2}\dotsm (1-q^{n_{1}})^2}
	,\\
	\sum_{n = 1}^{\infty}\left(\eta_{2k}^{\X{42}}(n)-\eta_{2k}^{\X{41}}(n)\right)q^{n} 
	&= 
		\sum_{n_{k} \geq \dotsb \geq n_{1} \geq 1} 
		\frac{(-q^2;q^2)_{n_{1}-1} q^{n_{1} + \dotsb + n_{k}}}
		{(q^{n_{1}+1})_{n_{1}} (q^{n_{1}})_{\infty}
			(1-q^{n_{k}})^{2}\dotsm (1-q^{n_{2}})^2}
	,\\
	\sum_{n = 1}^{\infty}\left(\eta_{2k}^{\X{42}}(n)-\eta_{2k}^{\X{39}}(n)\right)q^{n} 
	&= 
		\sum_{n_{k} \geq \dotsb \geq n_{1} \geq 1} 
		\frac{(-q^2;q^2)_{n_{1}-1} q^{2n_{1} +n_2+ \dotsb + n_{k}}}
		{(q^{n_{1}+1})_{n_{1}} (q^{n_{1}+1})_{\infty}
			(1-q^{n_{k}})^{2}\dotsm (1-q^{n_{1}})^2}
	,\\
	\sum_{n = 1}^{\infty}\left(\eta^{C2}_{2k}(n)-\eta_{2k}^{C1}(n)\right)q^{n} 
    &= 
    	\sum_{n_{k} \geq \dotsb \geq n_{1} \geq 1}
    	\frac{q^{n_{1} + n_{2} + \dotsb + n_{k}}}
    	{(q;q^2)_{n_{1}} (q^{n_{1}})_{\infty}  (1-q^{n_{k}})^{2}\dotsm (1-q^{n_{2}})^{2}}
	,\\
	\sum_{n = 1}^{\infty}\left(\eta_{2k}^{\X{39}}(n)-\eta_{2k}^{\X{38}}(n)\right)q^{n} 
	&= 
		\sum_{n_{k} \geq \dotsb \geq n_{1} \geq 1} 
		\frac{(-1;q^{2})_{n_{1}}q^{n_{1} + \dotsb + n_{k}}}
		{(q^{n_{1} + 1})_{n_{1}} (q^{n_{1}})_{\infty} (1-q^{n_{k}})^{2}\dotsm (1-q^{n_{2}})^2}
	,\\
	\sum_{n = 1}^{\infty}\left(\eta_{2k}^{\X{41}}(n)-\eta_{2k}^{\X{38}}(n)\right)q^{n} 
	&= 
		\sum_{n_{k} \geq \dotsb \geq n_{1} \geq 1} 
		\frac{(-q^2;q^2)_{n_{1}-1} q^{n_{1} + \dotsb + n_{k}}}
		{(q^{n_{1}+1})_{n_{1}} (q^{n_{1}+1})_{\infty}
			(1-q^{n_{k}})^{2}\dotsm (1-q^{n_{1}})^2}
	,\\
	\sum_{n = 1}^{\infty}\left(\eta^{C5}_{2k}(n)-\eta_{2k}^{C1}(n)\right)q^{n} 
    &= 
    	\sum_{n_{k} \geq \dotsb \geq n_{1} \geq 1}
    	\frac{q^{n_{1} + n_{2} + \dotsb + n_{k}} }
    	{ (q^{n_{1}})_{\infty}  (1-q^{n_{k}})^{2}\dotsm (1-q^{n_{2}})^{2}}    	
    	\times
    	\frac{(1-q^{\frac{n_1(n_1-1)}{2}})}{(q;q^2)_{n_{1}}(1-q^{n_1})}
	,\\
	\sum_{n = 1}^{\infty}\left(\eta^{C1}_{2k}(n)-\eta_{2k}^{L5}(n)\right)q^{n} 
    &= 
    	\sum_{n_{k} \geq \dotsb \geq n_{1} \geq 1}
    	\frac{q^{n_{1} + \dotsb + n_{k}}\left( \aprod{-1}{n_1}-1  \right)}
    	{(q;q^2)_{n_{1}} (q^{n_{1}+1})_{\infty}  (1-q^{n_{k}})^{2}\dotsm (1-q^{n_{1}})^{2}}
	,\\
	\sum_{n=1}^\infty \left(\eta^{J3}_{2k}(n)-\eta^{J2}_{2k}(n)\right)q^n
	&=
    	\sum_{n_{k} \geq \dotsb \geq n_{1} \geq 1}
    	\frac{q^{n_{1} + n_{2} + \dotsb + n_{k}}}
    	{(q)_{n_{1}-1} (q^{n_{1}+1})_{n_{1}} (q^{3n_1};q^3)_\infty (q^{n_{1}})_{\infty}
    		(1-q^{n_{k}})^{2}\dotsm(1-q^{n_{2}})^{2} }
	,\\	
	\sum_{n=1}^\infty \left(\eta^{J2}_{2k}(n)-\eta^{J1}_{2k}(n)\right)q^n
	&=
    	\sum_{n_{k} \geq \dotsb \geq n_{1} \geq 1}
    	\frac{q^{2n_{1} + n_{2} + \dotsb + n_{k}}}
    	{(q)_{2n_{1}} (q^{3n_1};q^3)_\infty (q^{n_{1}})_{\infty}
    		(1-q^{n_{k}})^{2}\dotsm(1-q^{n_{2}})^{2}}
	,\\
	\sum_{n=1}^\infty \left(\eta^{E4}_{2k}(n)-\overline{\eta}_{2k}(n)\right)q^n
	&=
		 \sum\limits_{n_k \geq \dotsc \geq n_1 \geq 1} 
		 \frac{\aprod{-q^{n_1+1}}{\infty}  q^{n_1+\dotsb+n_k}   } 
		 {(q^{n_1})_{\infty} (1-q^{n_k})^{2}\dotsm(1-q^{n_2})^2 }
	,\\
	\sum_{n = 1}^{\infty}\left(\overline{\eta}_{2k}(n)-\eta_{2k}^{\X{46}}(n)\right)q^{n} 
	&= 
		\sum_{n_{k} \geq \dotsb \geq n_{1} \geq 1} 
		\frac{(-q^{n_{1} + 1})_{\infty} q^{n_1+ \dotsb + n_k}}
		{(q^{n_1+1})_{\infty} (1-q^{n_{k}})^{2} \dotsm (1-q^{n_{1}})^{2}}
		\times
		\left(\frac{(-q^3;q^3)_{n_{1}-1} }{(q^{n_{1} + 1})_{n_{1}-1} }-1\right)
	,\\
	\sum_{n = 1}^{\infty}\left(\overline{\eta}_{2k}(n)-\eta_{2k}^{I14}(n)\right)q^{n} 
	&= 
		\sum_{n_{k} \geq \dotsb \geq n_{1} \geq 1} 
		\frac{(-q^{n_{1}+1})_{\infty} q^{n_{1} + \cdots{} + n_{k}}}
		{(q^{n_{1} + 1})_{\infty}
			(1-q^{n_{k}})^{2}\dotsm (1-q^{n_{1}})^{2}}
		\\&\quad
		\times
		\left(
		\frac{(-q^2;q^2)_{n_{1}-1}(1+q^{n_1})  }{(q;q^{2})_{n_{1}} } - 1
		\right)	
	,\\		
	\sum_{n = 1}^{\infty}\left(\eta^{G3}_{2k}(n)- \eta_{2k}^{G1}(n)\right)q^{n} 
    &=  
    	\sum\limits_{n_k \geq \dotsb \geq n_1 \geq 1} 
    	\frac{(-q^{2n_1+1};q^2)_\infty q^{2n_1+\dotsb +2n_k}} 
    	{(q^4;q^4)_{n_1} (q^{2n_1+2};q^2)_{\infty} (q^{2n_1};q^2)_{\infty}
    		(1-q^{2n_k})^2\dotsm (1-q^{2n_2})^2}
	,\\
	\sum_{n = 1}^{\infty}\left(\eta_{2k}^{Y2}(n) - \eta_{2k}^{Y1}(n)\right)q^{n} 
	&= 
		\sum_{n_{k} \geq \dotsb \geq n_{1} \geq 1}
		\frac{q^{n_{1}^2 + 2n_{2} + \dotsb + 2n_{k}}}
		{(q;q^2)_{n_{1}} (q^4;q^4)_{n_{1}}  (q^{2n_{1} + 2} ; q^2)_{\infty}(q^{2n_{1}} ; q^2)_{\infty} }
		\\&\quad
		\times
		\frac{1}{(1-q^{2n_{k}})^{2} \dotsm (1-q^{2n_{2}})^2}
	,\\
	\sum_{n = 1}^{\infty}\left(\eta_{2k}^{Y2}(n) - \eta_{2k}^{Y3}(n)\right)q^{n} 
	&= 
		\sum_{n_{k} \geq \dotsb \geq n_{1} \geq 1}
		\frac{q^{2n_{1} + \dotsb + 2n_{k}}}
		{(q^4;q^4)_{n_{1}}  (q^{2n_{1}} ; q^2)_{\infty} (q^{2n_{1} + 2} ; q^2)_{\infty} 
			(1-q^{2n_{k}})^{2} \dotsm (1-q^{2n_{2}})^2}
		\\&\quad
		\times 
		\frac{(1-q^{n_1^2})}
		{\aqprod{q}{q^2}{n_1}(1-q^{2n_1})}
	,\\	
	\sum_{n = 1}^{\infty}\left(\eta_{2k}^{Y4}(n) - \eta_{2k}^{Y3}(n)\right)q^{n} 
	&= 
		\sum_{n_{k} \geq \dotsb \geq n_{1} \geq 1}
		\frac{q^{2n_{1} + \dotsb + 2n_{k}}}
		{(q;q^2)_{n_{1}} (q^4;q^4)_{n_{1}}  (q^{2n_{1} + 2} ; q^2)_{\infty}(q^{2n_{1}} ; q^2)_{\infty}}
		\\&\quad
		\times
		\frac{1}{(1-q^{2n_{k}})^{2} \dotsm (1-q^{2n_{2}})^2}		
	.		
\end{align*}
Here $\overline{\eta}_{2k}$ is the symmetrized overpartition rank moment from
\cite{JenningsShaffer3} and satisfies
\begin{align*}
	\sum_{n = 1}^{\infty}\left(\overline{\mu}_{2k}(n)-\overline{\eta}_{2k}(n)\right)q^{n} 
	&= 
		\sum_{n_{k} \geq \dotsb \geq n_{1} \geq 1} 
		\frac{(-q^{n_{1}+1})_{\infty} q^{n_{1} + \cdots{} + n_{k}}}
		{(q^{n_{1} + 1})_{\infty} (1-q^{n_{k}})^{2}\dotsm (1-q^{n_{1}})^{2}}
.	
\end{align*}
\end{corollary}

\section{The Proof of Theorem \ref{TheoremMain}}

\begin{lemma}\label{LemmaRankBilateralAndDerivative}
Suppose $\alpha_n$ is a sequence such that $\alpha_{n} = \alpha_{-n}$, then 
\begin{align*}
	\sum_{n = 1}^{\infty} \frac{\alpha_{n}q^n(1-z)(1-z^{-1})}{(1-zq^n)(1-z^{-1}q^n)}
	&=
		\sum_{n \neq 0}\frac{\alpha_{n}q^n(1-z)}{(1+q^{n})(1-zq^{n})}
	.
\end{align*}
Furthermore, if $j$ is a positive integer, then
\begin{align*}
	\left(\frac{\partial}{\partial{}z}\right)^j 
	\sum_{n = 1}^{\infty} \frac{\alpha_{n}q^n(1-z)(1-z^{-1})}{(1-zq^n)(1-z^{-1}q^n)}
	&= 
		-j!\sum_{n \neq 0}
		\frac{\alpha_{n}q^{nj}(1-q^n)}{(1+q^n)(1-zq^n)^{j+1}}
	.
\end{align*}
\end{lemma}
\begin{proof}
We note the first identity is the standard rearrangements used
for many rank functions. 
We have that
\begin{align*}
	\sum_{n = 1}^{\infty} 
	\frac{\alpha_{n}q^n(1-z)(1-z^{-1})}{(1-zq^n)(1-z^{-1}q^n)} 
	&= 
		\sum_{n = 1}^{\infty} 
		\frac{\alpha_{n}q^n}{1 + q^n}
		\left(\frac{1-z}{1-zq^n} + \frac{1-z^{-1}}{1-z^{-1}q^n}\right) 
	\\
	&= 
		\sum_{n = 1}^{\infty} 
		\frac{\alpha_{n}q^n(1-z)}{(1 + q^n)(1-zq^n)} 
		+ 
		\sum_{n = -\infty}^{-1}
		\frac{\alpha_{-n}q^{-n}(1-z^{-1})}{(1 + q^{-n})(1-z^{-1}q^{-n})} 
	\\
	&= 
		\sum_{n = 1}^{\infty} 
		\frac{\alpha_{n}q^n(1-z)}{(1 + q^n)(1-zq^n)} 
		+ 
		\sum_{n = -\infty}^{-1}
		\frac{\alpha_{n}q^n(1-z)}{(1 + q^{n})(1-zq^{n})} 
	\\
	&=
		\sum_{n \neq 0}\frac{\alpha_{n}q^n(1-z)}{(1+q^{n})(1-zq^{n})}
	.
\end{align*}
This establishes the first identity. The second identity follows from the first,
upon noting that 
\begin{align*}
	\left(\frac{\partial}{\partial{}z}\right)^j 
	\frac{1-z}{1-zq^n}
	&=
	\frac{-j!(1-q^n)q^{n(j-1)}}{(1-zq^n)^{j+1}}	
.
\end{align*}	
\end{proof}

\begin{lemma}\label{LemmaSymmetrizedMoments}
Suppose $\alpha_n$ is a sequence such that $\alpha_{n} = \alpha_{-n}$ and
\begin{align*}
	R_{X}(z,q)
	:=
	P_{X}(q)
	\left(
		1
		+
		\sum_{n = 1}^{\infty} \frac{\alpha_{n}q^n(1-z)(1-z^{-1})}{(1-zq^n)(1-z^{-1}q^n)}
	\right)	
	&=
	\sum_{n=0}^\infty \sum_{m=-\infty}^\infty
		N_{X}(m,n) z^mq^n
	,
\end{align*}
where $P_{X}(q)$ is some series in $q$.
Let
\begin{align*}
	\eta^{X}_{k}(n)
	&= 
		\sum_{m = -\infty}^{\infty}
		\binom{m+\lfloor\frac{k-1}{2}\rfloor}{k}N_{X}(m,n)
	.
\end{align*}
Then $\eta^{X}_{2k+1}(n)=0$ and 
\begin{align*}
	\sum_{n=1}^\infty \eta^X_{2k}(n) q^n
	&=
		-P_X(q)
		\sum_{n=1}^\infty
		\frac{\alpha_{n}q^{nk}}{(1-q^n)^{2k}}
	.
\end{align*}	
\end{lemma}
\begin{proof}
We note this is the generalization of Theorems 1 and 2 of \cite{Andrews4}.
The proof that $\eta^{X}_{2k+1}=0$ follows that of Theorem 1 of \cite{Andrews4}
verbatim, as we have $N_X(m,n)=N_X(-m,n)$ due to the symmetry in $z$ and 
$z^{-1}$. For the even moments, much the same as in the proof of Theorem 2 from 
\cite{Andrews3}, by Lemma \ref{LemmaRankBilateralAndDerivative} we have that
\begin{align*}
	\sum\limits_{n=1}^\infty \eta^X_{2k}(n)q^n 
    &= 
    	\frac{1}{(2k)!}
    	\left(\frac{\partial}{\partial{}z}\right)^{2k}
    	z^{k-1}R_X(z,q) \Bigr|_{z = 1} 
    \\
	&= 
		\frac{1}{(2k)!}
		\sum_{j = 0}^{k-1}
		\binom{2k}{j}(k-1)\cdots{}(k-j)R_X^{(2k-j)}(1,q) 
	\\
	&= 
		-P_X(q)
		\sum_{j=0}^{k-1} \binom{k-1}{j}
		\sum_{n \neq 0}
		\frac{\alpha_{n}q^{n(2k-j)}(1-q^n)}{(1+q^n)(1-q^n)^{2k-j+1}} 
	\\
	&= 
		-P_X(q)
		\sum_{n \neq 0}\frac{\alpha_{n}q^{2nk}}{(1 - q^n)^{2k}(1+q^n)}
		\sum_{j = 0}^{k-1}\binom{k-1}{j}(q^{-n}(1 - q^{n}))^{j} 
	\\
	&= 
		-P_X(q)
		\sum_{n \neq 0}
		\frac{\alpha_{n}q^{2nk}}{(1 - q^n)^{2k}(1+q^n)}( 1 + q^{-n}(1 - q^{n}))^{k - 1} 
	\\
	&= 
		-P_X(q)
		\sum_{n \neq 0}
		\frac{\alpha_{n}q^{n + nk}}{(1 - q^n)^{2k}(1+q^n)} 
	\\
	&= 
		-P_X(q)
		\sum_{n = 1}^{\infty}
		\frac{\alpha_{n}q^{nk}}{(1-q^n)^{2k}}
	.
\end{align*}
\end{proof}
With the above lemma, we have established (\ref{EqMainTheoremSymmetrizedRankMoment})
of Theorem \ref{TheoremMain}. We then find 
(\ref{EqMainTheoremSymmetrizexRankCrankDifference}) follows immediately from 
equation (3.3) and Theorem 3.3 of \cite{Garvan1}, both of which we state for 
completeness.
Equation (3.3) is that for positive $k$,
\begin{align*}
	\sum_{n=1}^\infty \mu_{2k}(n)q^n
	&=
	\frac{1}{\aprod{q}{\infty}}
	\sum_{n_k\geq \dotsb \geq n_1\geq 1}
	\frac{q^{n_1+\dotsb+n_k}}{(1-q^{n_k})^2\dotsm (1-q^n_1)^2}	
.
\end{align*}
Theorem 3.3 states that if $\alpha_{n}$ and $\beta_{n}$ are a Bailey pair 
relative to $(1,q)$ and  $\alpha_{0} = \beta_{0} = 1$. Then
\begin{align*}
	&\sum_{n_{k} \geq \dotsb \geq n_{1} \geq 1} 
	\frac{(q )_{n_{1}}^{2} q^{n_{1} + \dotsb + n_{k}}\beta_{n_{1}}}
	{(1-q^{n_{k}})^{2}(1-q^{n_{k-1}})^{2} \dotsm (1-q^{n_{1}})^{2}} 
	\\
	&= 
		\sum_{n_{k} \geq \dotsb \geq n_{1} \geq 1} 
		\frac{q^{n_{1} + \dotsb + n_{k}}}
		{(1-q^{n_{k}})^{2}(1-q^{n_{k-1}})^{2} \dotsm (1-q^{n_{1}})^{2}} 
		+ 
		\sum_{r = 1}^{\infty}\frac{q^{kr}\alpha_{r}}{(1-q^{r})^{2k}}
	.
\end{align*}
Lastly, for (\ref{EqMainTheoremSymmetrizedRankEquality}) and
(\ref{EqMainTheoremOrdinaryRankEquality}), we find that one can follow verbatim 
the proof of Theorem 4.3 of \cite{Garvan1}, as the only requirement is that
$N_X(m,n)=N_{X}(-m,n)$. Thus Theorem \ref{TheoremMain} follows.

\section{Combinatorial Interpretations}

To determine what the functions $\hspt{X}{2k}{n}$ are counting, we use the weight
from \cite{Garvan1} as extended to vector partitions in \cite{JenningsShaffer3}.
We recall that a vector partition of $n$ is a vector of partitions, 
$(\pi_1,\dotsc,\pi_r)$, such that altogether the parts sum to $n$. When
$\vec{\pi}$ is a vector partition of $n$, we write $|\vec{\pi}|=n$.
For a partition $\pi$ with different parts $n_1 < n_2 < \cdots{} < n_m$, we 
take $f_{j}(\pi)$ to be the frequency of the part $n_j$.
Given a vector partition $\vec{\pi}=(\pi_1,\pi_2,\dotsc,\pi_r)$,
we let $f^1_{j} := f^1_{j}(\vec{\pi}) = f_j(\pi_1)$,
and have the weight $\omega_k(\vec{\pi})$ given by
\begin{align*}
	\omega_{k}(\vec{\pi}) 
	&:= 
		\sum_{\substack{m_1 + m_2 + \cdots{} + m_r = k \\ 1 \leq r \leq k}}
		\binom{f_{1}^{1} + m_1 - 1}{2m_1 - 1} 
    	\sum_{2 \leq j_2 < j_3 < \cdots < j_r}
    	\binom{f_{j_2}^{1} + m_2}{2m_2}\binom{f_{j_3}^{1} + m_3}{2m_3}\cdots{}\binom{f_{j_r}^{1} + m_r}{2m_r}
.
\end{align*}
We note that the first sum is over all compositions of $k$.
Based on the combinatorial interpretations in \cite{Garvan1, JenningsShaffer3}, 
we expect to find that $\hspt{X}{2k}{n}$ is the number of vector partitions
$\vec{\pi}=(\pi_1,\pi_2,\dotsc,\pi_r)$, weighted by $\omega_{k}$, where
$(\pi_2,\dotsc,\pi_r)$ are restricted to the vector partitions enumerated
by $\aprod{q^{n_1+1}}{\infty}\aprod{q}{n_1}^2 P_X(q) \beta_{n_1}$, where $n_1$ is the smallest 
part of $\pi_1$. Upon minor adjustments for the cases where 
$q\rightarrow q^2$ and $q\rightarrow q^4$, this is correct. In a small number 
of cases one of the partitions $\pi_k$, for $k\geq 2$, strictly speaking
may not be a partition in that we will allow a non-positive part to appear.

\begin{definition}
We define the follows sets of vector partitions. In all cases we require that
$\pi_1$ be non-empty.
\begin{itemize}
\item $S^{A1}$ - the set of vector partitions $(\pi_1, \pi_2)$ 
where $\pi_1$ has smallest part $n$ and the parts $m$ of $\pi_2$ 
satisfy $n+1 \leq m \leq 2n$. 
        
\item $S^{A3}$ - the set of vector partitions $(\pi_1, \pi_2)$ 
where $\pi_1$ has smallest part $n$; the parts $m$ of $\pi_2$ satisfy 
$n \leq m \leq 2n$; and the part $n$ appears exactly once in $\pi_2$.
    
\item $S^{A5}$ - the set of vector partitions $(\pi_1, \pi_2)$ 
where $\pi_1$ has smallest part $n$; the parts $m$ of $\pi_2$ satisfy 
$n\leq m \leq 2n$; and the part $n$ appears exactly $n$ times in $\pi_2$.
    
\item $S^{A7}$ - the set of vector partitions $(\pi_1, \pi_2)$ 
where $\pi_1$ has smallest part $n$; the parts $m$ of $\pi_2$ satisfy 
$n \leq m \leq 2n$; and the part $n$ appears exactly $n-1$ times in $\pi_2$.
    
\item $S^{B2}$ - the set of vector partitions $(\pi_1, \pi_2)$ 
where $\pi_1$ has smallest part $n$ and $\pi_2$ consists of exactly
one copy of the part $n$.
        
\item $S^{C1}$ - the set of vector partitions $(\pi_1, \pi_2)$ where the 
smallest part of $\pi_1$ is $n$ and 
the parts $m$ of $\pi_2$ are odd and satisfy $1\leq m\le 2n-1$.
    
\item $S^{C2}$ - the set of vector partitions $(\pi_1, \pi_2)$ 
where the smallest part of $\pi_1$ is $n$;
the parts $m$ of $\pi_2$ satisfy $1\leq m\leq 2n-1$;
the parts of $\pi_2$ other than $n$ are odd;
if $n$ is odd, then $n$ appears at least once in $\pi_2$;
and if $n$ is even, then $n$ appears exactly once in $\pi_2$.
    
\item $S^{C5}$ - the set of vector partitions $(\pi_1, \pi_2, \pi_3)$ 
where the smallest part of $\pi_1$ is $n$;
the parts $m$ of $\pi_2$ are odd and satisfy $1\leq m\leq 2n-1$;
and $\pi_3$ consists of exactly one copy of the part $\frac{n(n - 1)}{2}$.        
        
\item $S^{L5}$ - the set of vector partitions $(\pi_1, \pi_2, \pi_3)$ 
where $\pi_1$ has smallest part $n$;
the parts $m$ of $\pi_2$ are odd and satisfy $1\leq m\leq 2n-1$;
and the parts $m$ of $\pi_3$ are distinct and satisfy  
$0\leq m \leq n-1$ (so that $\pi_3$ may contain $0$ at most once). 
    
\item $S^{\X{38}}$ - the set of vector partitions $(\pi_1, \pi_2, \pi_3)$ 
where $\pi_1$ has smallest part $n$;
the parts $m$ of $\pi_2$ satisfy $n+1\leq m\leq 2n$;
and the parts $m$ of $\pi_3$ are distinct even parts and satisfy
$0\leq m\leq 2n-2$ (so that $\pi_3$ may contain the part $0$ at most once).
    
\item $S^{\X{39}}$ - the set of vector partitions $(\pi_1, \pi_2, \pi_3)$ 
where $\pi_1$ has smallest part $n$;
the parts $m$ of $\pi_2$ satisfy $n\leq m\leq 2n$;
the part $n$ appears exactly once in $\pi_2$;
and the parts $m$ of $\pi_3$ are distinct even parts and satisfy
$0\leq m\leq 2n-2$ (so that $\pi_3$ may contain the part $0$ at most once).
    
\item $S^{\X{41}}$ - the set of vector partitions $(\pi_1, \pi_2, \pi_3)$ 
where $\pi_1$ has smallest $n$;
the parts $m$ of $\pi_2$ satisfy $n+1\leq m\leq 2n$;
and the parts $m$ of $\pi_3$ are distinct even parts and satisfy
$1\leq m\leq 2n-2$.
    
\item $S^{\X{42}}$ - the set of vector partitions $(\pi_1, \pi_2, \pi_3)$ 
where $\pi_1$ has smallest $n$;
the parts $m$ of $\pi_2$ satisfy $n\leq m\leq 2n$;
the part $n$ appears exactly once in $\pi_2$;
and the parts $m$ of $\pi_3$ are distinct even parts and satisfy
$1\leq m\leq 2n-2$.
    
\item $S^{E4}$ - the set of vector  partitions $(\pi_1, \pi_2)$ 
where $\pi_1$ has smallest part $n$; 
the parts $m$ of $\pi_2$ are distinct and satisfy
$n\leq m$; and the part $n$ appears exactly once in $\pi_2$.    

\item $S^{I14}$ - the set of vector partitions $(\pi_1, \pi_2, \pi_3)$ 
where $\pi_1$ has smallest part $n$; 
the parts $m$ of $\pi_2$ satisfy $1\leq m\leq 2n-1$;
the even parts of $\pi_2$ are distinct;
and the parts $m$ of $\pi_3$ are distinct and satisfy
$n\leq m$.

\item $S^{\X{46}}$ - the set of vector partitions $(\pi_1, \pi_2, \pi_3, \pi_4)$ 
where $\pi_1$ has smallest part $n$;
the parts $m$ of $\pi_2$ are distinct multiples of $3$ and satisfy $1\leq m\leq 3n-3$;
the parts $m$ of $\pi_3$ are distinct and satisfy $n+1\leq m$;
and the parts $m$ of $\pi_3$ satisfy $n+1\leq m \leq 2n-1$.
    
\item $S^{J1}$ - the set of  vector partitions $(\pi_1, \pi_2)$ 
where $\pi_1$ has smallest part $n$;
$\pi_2$ has no part $m$ such that $2n\leq m\leq 3n-1$;
and the parts $m$ of $\pi_2$ satisfying $3n\leq m$ are
multiplies of $3$.
    
\item $S^{J2}$ - the set of vector partitions $(\pi_1, \pi_2)$ 
where $\pi_1$ has smallest part $n$;
$\pi_2$ has no part $m$ such that $m=n$ or $2n+1\leq m\leq 3n-1$;
and the parts $m$ of $\pi_2$ satisfying $3n\leq m$ are multiples of $3$.
        
\item $S^{J3}$ - the set of vector partitions $(\pi_1, \pi_2)$ 
where $\pi_1$ has smallest part $n$;
$\pi_2$ has no part $m$ such that $2n+1\leq m\leq 3n-1$;
the part $n$ appears exactly once in $\pi_2$;
and the parts $m$ of $\pi_2$ satisfying $3n\leq m$ are multiples of $3$.
        
\item $S^{\X{40}}$ - the set of vector partitions $(\pi_1, \pi_2, \pi_3)$ 
where $\pi_1$ has smallest part $n$;
the parts $m$ of $\pi_2$ satisfy $1\leq m\leq n-1$;
the parts $m$ of $\pi_3$ satisfying $1\leq m\leq 2n-1$ are odd;
and the parts $m$ of $\pi_3$ satisfying $2n\leq m$ are even.

\item $S^{F3}$ - the set of vector partitions $(\pi_1, \pi_2)$ 
where the parts of $\pi_1$ are even and the smallest part is $2n$; 
$\pi_2$ contains exactly one copy of the negative integer $-n$; and the
positive parts $m$ of $\pi_2$ are odd and satisfy $1\leq m\leq 2n-1$.
    
\item $S^{G1}$ - the set of vector partitions $(\pi_1, \pi_2, \pi_3)$ 
where the parts of $\pi_1$ are even and the smallest part is  $2n$;
the parts $m$ of $\pi_2$ are multiples of $4$ and satisfy $1\leq m\leq 4n$;
the parts $m$ of $\pi_3$ satisfy $2n+1\leq m$;
and the odd parts of $\pi_3$ are distinct.
    
\item $S^{G3}$ - the set of vector partitions $(\pi_1, \pi_2, \pi_3)$ 
where the parts of $\pi_1$ are even and the smallest part is  $2n$;
the parts $m$ of $\pi_2$ are multiples of $4$ and satisfy $1\leq m\leq 4n$;
the parts $m$ of $\pi_3$ satisfy $2n\leq m$;
the odd parts of $\pi_3$ are distinct;
and the part $2n$ appears exactly once in $\pi_3$.
     
\item $S^{Y1}$ - the set of vector partitions $(\pi_1,\pi_2,\pi_3,\pi_4)$ 
where the parts of $\pi_1$ are even and the smallest part is $2n$;
the parts $m$ of $\pi_2$ are multiples of $4$ and satisfy $1\leq m\leq 4n$;
the parts $m$ of $\pi_3$ satisfying $1\leq m\leq 2n-1$ are odd;
the parts $m$ of $\pi_3$ satisfying $2n+2\leq m$ are even;
the parts $2n$ and $2n+1$ do not appear in $\pi_3$;
and $\pi_4$ consists of exactly one copy of the part $n^2-2n$ (which
in the case of $n=1$ is $-1$).
 
\item $S^{Y2}$ - the set of vector partitions $(\pi_1,\pi_2,\pi_3,\pi_4)$ 
where the parts of $\pi_1$ are even and the smallest part is $2n$;
the parts $m$ of $\pi_2$ are multiples of $4$ and satisfy $1\leq m\leq 4n$;
the parts $m$ of $\pi_3$ satisfying $1\leq m\leq 2n-1$ are odd;
the parts $m$ of $\pi_3$ satisfying $2n+2\leq m$ are even;
the parts $2n$ and $2n+1$ do not appear in $\pi_3$;
and $\pi_4$ consists of exactly one copy of the part $n^2$.
    
\item $S^{Y3}$ - the set of vector partitions $(\pi_1,\pi_2,\pi_3)$ 
where the parts of $\pi_1$ are even and the smallest part is $2n$;
the parts $m$ of $\pi_2$ are multiples of $4$ and satisfy $1\leq m\leq 4n$;
the parts $m$ of $\pi_3$ satisfying $1\leq m\leq 2n-1$ are odd;
the parts $m$ of $\pi_3$ satisfying $2n+2\leq m$ are even;
and the parts $2n$ and $2n+1$ do not appear in $\pi_3$.
    
\item $S^{Y4}$ - the set of vector partitions $(\pi_1,\pi_2,\pi_3)$ 
where the parts of $\pi_1$ are even and the smallest part is $2n$;
the parts $m$ of $\pi_2$ are multiples of $4$ and satisfy $1\leq m\leq 4n$;
the parts $m$ of $\pi_3$ satisfying $1\leq m\leq 2n-1$ are odd;
the parts $m$ of $\pi_3$ satisfying $2n\leq m$ are even;
and the part $2n$ appears exactly once in $\pi_3$.    
    
\item $S^{L2}$ - the set of vector partitions $(\pi_1, \pi_2)$ 
where the parts of $\pi_1$ are multiplies of $4$ and the smallest part is $4n$;
the parts $m$ of $\pi_2$ satisfy $m\not\equiv 2\pmod{4}$ and
$4n+1\leq m$;
and the parts $m$ of $\pi_2$ that are multiples of $4$ additionally satisfy
$m\leq 8n$.
\end{itemize}
\end{definition}

\begin{theorem}
Suppose the assumptions and notation of Theorem \ref{TheoremMain}
and let
\begin{align*}
	\beta'_{n}(q) 
	&= 
		(q^{n + 1})_{\infty}(q)_{n}^{2} P_X(q) \beta_{n}
.
\end{align*}
Suppose $k$ is a positive integer and let $A$ denote the set of all compositions
of $k$, then
\begin{align*}
	\sum_{n=1}^{\infty} \hspt{X}{k}{n}q^n 
	&= 
		\sum_{(m_1,\dotsc,m_r)=\vec{m} \in A} 
		\hspace{.5em}
		\sum_{1 \leq n_{1} < n_{j_2}< \dotsb <  n_{j_r}} 
		\sum_{f_1=m_1}^{\infty} \sum_{f_{j_2}=m_2}^{\infty} 
		\dotsm \sum_{f_{j_r}=m_r}^{\infty} 
		\binom{f_{1}+m_1-1}{2m_1-1} 
		\\&\quad	
		\times  \binom{f_{j_2}+m_2}{2m_2} \dotsm 
		\binom{f_{j_r}+m_r}{2m_r} q^{n_1f_1+n_{j_2}f_{j_2}+ \dotsb +n_{j_r}f_{j_r}}
		\beta_{n_1}'(q) \prod_{\substack{i > n_1 \\ i \notin \{n_{j_2},\dotsc ,n_{j_r}\}}} \frac{1}{1 - q^i}
	.
\end{align*} 

In particular for all of the Bailey pairs considered in Corollary 
\ref{CorollarySeriesIdentities},  
we have for all positive $k$ and $n$ that
\begin{align*}
	\sum_{n=1}^\infty
	\hspt{X}{k}{n}q^n
	&= 
		\sum_{\vec{\pi} \in S^{X}}\omega_{k}(\vec{\pi}) q^{|\vec{\pi}|}
	.
\end{align*}
\end{theorem}
\begin{proof}
The proof is much the same as that of Theorem 5.6 from \cite{Garvan1} and its
generalization to Theorem 3.1 from \cite{JenningsShaffer3}.
We first note that 
\begin{align}\label{EqTheoremCombinatorics}
	\sum_{n=1}^{\infty} \hspt{X}{k}{n}q^n 
	&= 
	\sum_{1 \leq n_1 \leq n_2 \leq \dotsb \leq n_k } 
	\frac{ \beta'_{n_1}(q) q^{n_1+n_2+\dotsb +n_k}} 
	{(q^{n_1+1})_{\infty}(1-q^{n_k})^2\dotsm (1-q^{n_1})^2}
	.
\end{align}
We will use that
\begin{align*}
	\frac{x^j}{(1-x)^{2j}} 
	&=
		\sum_{n = j}^{\infty}\binom{n + j - 1}{2j - 1}x^n 
	,&
	\frac{x^j}{(1-x)^{2j+1}}
	&= 
		\sum_{n = j}^{\infty}\binom{n + j}{2j}x^n 
	.
\end{align*}

To illustrate the series rearrangements of the general proof, we first write 
out in full detail the case when $k = 3$. For $k=3$ we have
\begin{align*}
	&\sum_{n=1}^{\infty} \hspt{X}{3}{n}q^n 
	= 
		\sum_{n_{j_3} \geq n_{j_2} \geq n_{1} \geq 1}
		\frac{\beta_{n_1}'(q)q^{n_{1} + n_{j_2}  + n_{j_3}}}
		{(q^{n_{1} + 1})_{\infty}(1-q^{n_{j_3}})^{2}(1-q^{n_{j_2}})^{2}(1-q^{n_{1}})^{2}} 
	\\
	&= 
		\left(
		\sum_{1 \leq n_{1} = n_{j_2} =  n_{j_3}} 
		+ 
		\sum_{1 \leq n_{1} = n_{j_2} <  n_{j_3}} 
		+ 
		\sum_{1 \leq n_{1} < n_{j_2} =  n_{j_3}}   
		+
		\sum_{1 \leq n_{1} < n_{j_2} <  n_{j_3}}
		\right)
		\frac{\beta_{n_1}'(q)q^{n_{1} + n_{j_2}  + n_{j_3}}}
			{(q^{n_{1} + 1})_{\infty}(1-q^{n_{j_3}})^{2}(1-q^{n_{j_2}})^{2}(1-q^{n_{1}})^{2}}   
	\\
	&=
		\sum_{1 \leq n_1} 
		\frac{q^{3n_1}}{(1-q^{n_1})^6}\beta_{n_1}'(q)\prod_{i >n_1}\frac{1}{1-q^i}  
		+ 
		\sum_{1 \leq n_1 < n_{j_3}} 
		\frac{q^{2n_1}}{(1-q^{n_1})^4} \frac{q^{n_{j_3}}}{(1-q^{n_{j_3}})^3}
		\beta_{n_1}'(q)\prod_{\substack{i > n_1 \\ i \neq n_{j_3}}} \frac{1}{1 - q^i} 
		\\&\quad
		+ 
		\sum_{1 \leq n_1 < n_{j_2}} 
		\frac{q^{n_1}}{(1-q^{n_1})^2} \frac{q^{2n_{j_2}}}{(1-q^{n_{j_2}})^5}
		\beta_{n_1}'(q)\prod_{\substack{i > n_1 \\ i \neq n_{j_2}}} \frac{1}{1 - q^i}
		\\&\quad
		+
		\sum_{1 \leq n_1 < n_{j_2} <n_{j_3}} 
		\frac{q^{n_1}}{(1-q^{n_1})^2} \frac{q^{n_{j_2}}}{(1-q^{n_{j_2}})^3}\frac{q^{n_{j_3}}}{(1-q^{n_{j_3}})^3}
		\beta_{n_1}'(q)\prod_{\substack{i > n_1 \\ i \neq n_{j_2},n_{j_3}}} \frac{1}{1 - q^i}  
	\\
	&= 
		\sum_{1 \leq n_{1}} \sum_{f_1=3}^{\infty} 
		\binom{f_1+3-1}{6-1}q^{n_1f_1}\beta_{n_1}'(q)\prod_{i >n_1}\frac{1}{1-q^i}
		\\&\quad
		+  
		\sum_{1 \leq n_1 < n_{j_3}} \sum_{f_1=2}^{\infty} 
		\binom{f_1+2-1}{4-1}q^{n_1f_1}\sum_{f_{j_3}=1}^{\infty} 
		\binom{f_{j_3}+1}{2}q^{n_{j_3}f_{j_3}}
		\beta_{n_1}'(q)\prod_{\substack{i > n_1 \\ i \neq n_{j_3}}} \frac{1}{1 - q^i} 
		\\&\quad
		+  
		\sum_{1 \leq n_1 < n_{j_2}} \sum_{f_1=1}^{\infty} 
		\binom{f_1+1-1}{2-1}q^{n_1f_1}\sum_{f_{j_2}=2}^{\infty} 
		\binom{f_{j_2}+2}{4}q^{n_{j_2}f_{j_2}}
		\beta_{n_1}'(q)\prod_{\substack{i > n_1 \\ i \neq n_{j_2}}} \frac{1}{1 - q^i} 
		\\&\quad
		+  
		\sum_{1 \leq n_1 < n_{j_2}< n_{j_3}} \sum_{f_1=1}^{\infty} 
		\binom{f_1+1-1}{2-1}q^{n_1f_1}\sum_{f_{j_2}=1}^{\infty} 
		\binom{f_{j_2}+1}{2}q^{n_{j_2}f_{j_2}}\sum_{f_{j_3}=1}^{\infty} 
		\binom{f_{j_3}+1}{2}q^{n_{j_3}f_{j_3}}
		\beta_{n_1}'(q)
		\hspace{-.5em}		
		\prod_{\substack{i > n_1 \\ i \neq n_{j_2},n_{j_3}}} \frac{1}{1 - q^i}
	.
\end{align*}
Now the set of compositions of 3 is $A= \{(3),(2,1),(1,2),(1,1,1)\}$, and so we 
have that
\begin{align*}
	\sum_{n=1}^{\infty} \hspt{X}{3}{n}q^n 
	&= 
		\sum_{(m_1,\dotsc ,m_r)=\vec{m} \in A} 
		\hspace{.5em}
		\sum_{1 \leq n_{1} < n_{j_2}<\dotsb <  n_{j_r}} 
		\sum_{f_1=m_1}^{\infty} \sum_{f_{j_2}=m_2}^{\infty} \dotsm \sum_{f_{j_r}=m_r}^{\infty} 
		\binom{f_{1}+m_1-1}{2m_1-1} 
		\\&\quad
		\times  
		\binom{f_{j_2}+m_2}{2m_2} \dotsm \binom{f_{j_r}+m_r}{2m_r}
		q^{n_1f_1+n_{j_2}f_{j_2}+\dotsb +n_{j_r}f_{j_r}}\beta'(q) 
		\prod_{\substack{i > n_1 \\ i \notin \{n_{j_2},\dotsc ,n_{j_r}\}}} \frac{1}{1 - q^i}
	. 
\end{align*} 

For general $k$ we take the expression for 
$\sum_{n= 1}^{\infty} \hspt{X}{k}{n}q^{n}$ in (\ref{EqTheoremCombinatorics}) 
and split the sum into $2^{k-1}$ sums by turning the index bounds into $<$ 
or $=$, each of which corresponds to a composition of $k$. In particular
we recall that the $2^{k-1}$ compositions of $k$ can be obtained by
taking a list of $k$ ones and between each one we put either a plus or a
comma. Given the index bounds $n_1 \leq n_2 \leq \dotsb \leq n_k$,
we make a choice of each $\leq$ being ``$=$'' or ``$<$'';
we associate to this the composition 
$1^{+}\hspace{-.5em}, 1^{+}\hspace{-.5em},\dotsc  {{^{+}}1} \hspace{-1em},\hspace{1em}$
where in order we choose ``$+$'' when we chose ``$=$'' and choose ``$,$'' when we chose 
``$<$''.

If we let $A$ be the 
set of all compositions of $k$, with the manipulations discussed above, we 
have that 
\begin{align*}
	\sum_{n=1}^{\infty} \hspt{X}{k}{n}q^n 
	&= 
		\sum_{(m_1,\dotsc,m_r)=\vec{m} \in A} 
		\hspace{.5em}
		\sum_{1 \leq n_{1} < n_{j_2}< \dotsb <  n_{j_r}} 
		\sum_{f_1=m_1}^{\infty} \sum_{f_{j_2}=m_2}^{\infty} 
		\dotsm \sum_{f_{j_r}=m_r}^{\infty} 
		\binom{f_{1}+m_1-1}{2m_1-1} 
		\\&\quad	
		\times  \binom{f_{j_2}+m_2}{2m_2} \dotsm 
		\binom{f_{j_r}+m_r}{2m_r} q^{n_1f_1+n_{j_2}f_{j_2}+ \dotsb +n_{j_r}f_{j_r}}
		\beta_{n_1}'(q) \prod_{\substack{i > n_1 \\ i \notin \{n_{j_2},\dotsc ,n_{j_r}\}}} \frac{1}{1 - q^i}
	.
\end{align*} 
This we recognize as the generating function for vector partitions 
$\vec{\pi}=(\pi_1,...,\pi_r)$ counted according to the weight 
$\omega_k(\vec{\pi})$ where $\beta_{n_1}'(q)$ determines the types of 
partitions in $(\pi_2,\dotsc,\pi_r)$.
\end{proof}
We note that taking $k=1$ means that for $\hspt{X}{1}{n}$ we simply count the 
number of appearances of the smallest part in $\pi_1$, and many of these 
functions have been studied elsewhere.
In particular, those $\hspt{X}{1}{n}$ which possess simple
linear congruences were studied by Garvan and the second author in
\cite{GarvanJenningsShaffer2, JenningsShaffer1, JenningsShaffer2};
$\hspt{C1}{1}{n}$ was studied by Andrews, Dixit, and Yee in 
\cite{AndrewsDixitYee1}; and 
$\hspt{C1}{1}{n}$, $\hspt{C5}{1}{n}$, and $\hspt{J1}{1}{n}$
were studied by Patkowski in \cite{Patkowski1, Patkowski2}.

To demonstrate these weighted counts, we compute a table of values for
$\hspt{X}{k}{n}$ with $X=A1,B2$,  $k=1,2,3$, and $n=4,5$. 
We note that the first three weights are given by
\begin{align*}
	\omega_{1}(\pi) 
	&= 
		f_{1}^{1}(\pi) 
	,\\
	\omega_{2}(\pi) 
	&= 
		\binom{f_{1}^{1}(\pi) + 1}{3} 
		+ 
		f_{1}^{1}(\pi)\sum_{2 \leq j}\binom{f_{j}^{1}(\pi) + 1}{2} 
	,\\
	\omega_{3}(\pi) 
	&= 
		\binom{f_{1}^{1}(\pi) + 2}{5} 
		+ 
		\binom{f_{1}^{1}(\pi) + 1}{3}\sum_{2 \leq j}\binom{f_{j}^{1}(\pi) + 1}{2} 
		+ 
		f_{1}^{1}(\pi)\sum_{2 \leq j}\binom{f_j + 2}{4} 
		\\&\quad
		+ 
		f_{1}^{1}(\pi)\sum_{2 \leq j < k}\binom{f_{j}^{1}(\pi) + 1}{2}\binom{f_{k}^{1}(\pi) + 1}{2}
	.  
\end{align*}
\begin{align*}
\begin{array}{c}
\begin{array}{c|c|c|c|c|c|c}
  	\multicolumn{7}{c}{\mbox{T{\sc able} 2. }A1\mbox{ Partitions of }4} 
  	\\
	&f_{1}^1   & f_{2}^1   & f_{3}^1   & \omega_1   & \omega_2 & \omega_3 
	\\\hline
	(4, \varnothing) 			& 1 & 0 & 0 & 1 & 0 & 0 \\
	(3+1, \varnothing) 			& 1 & 1 & 0 & 1 & 1 & 0\\
	(2 + 2, \varnothing) 		& 2 & 0 & 0 & 2 & 1 & 0\\
	(2+1+1, \varnothing) 		& 2 & 1 & 0 & 2 & 3 & 1\\
	(1+1+1+1, \varnothing) 		& 4 & 0 & 0 & 4 & 10 & 6\\
	(1+1, 2) 					& 2 & 0 & 0 & 2 & 1 & 0\\
	\mbox{Total} &&&&12&16&7
\end{array}
\\\\\\\\\\\\
\end{array}
&\quad\quad
\begin{array}{c|c|c|c|c|c|c}
  	\multicolumn{7}{c}{\mbox{T{\sc able} 3. }A1\mbox{ Partitions of }5} 
	\\
   	& f_{1}^1              & f_{2}^1       & f_{3}^1          & \omega_1               & \omega_2 & \omega_3 
   	\\\hline
	(5, \varnothing) 			& 1 & 0 & 0 & 1 & 0 & 0 \\
	(4+1, \varnothing) 			& 1 & 1 & 0 & 1 & 1 & 0\\
	(3+2, \varnothing) 			& 1 & 1 & 0 & 1 & 1 & 0\\
	(3+1+1, \varnothing) 		& 2 & 1 & 0 & 2 & 3 & 1\\
	(2+2+1, \varnothing) 		& 1 & 2 & 0 & 1 & 3 & 1\\
	(2+1+1+1, \varnothing) 		& 3 & 1 & 0 & 3 & 7 & 5\\
	(1+1+1+1+1, \varnothing) 	& 5 & 0 & 0 & 5 & 20 & 21\\
	(2, 3) 						& 1 & 0 & 0 & 1 & 0 & 0\\
	(2+1, 2) 					& 1 & 1 & 0 & 1 & 1 & 0 \\
	(1+1+1, 2) 					& 3 & 0 & 0 & 3 & 4 & 1\\
	(1, 2+2) 					& 1 & 0 & 0 & 1 & 0 & 0\\
\mbox{Total} &&&&20&40&29
\end{array}
\\
\begin{array}{c|c|c|c|c|c|c}
	\multicolumn{7}{c}{\mbox{T{\sc able} 4. }B2\mbox{ Partitions of }4} 
	\\
	& f_{1}^1	& f_{2}^1      & f_{3}^1   	& \omega_1 & \omega_2 		& \omega_3
	\\\hline
	(2, 2) 			& 1 & 0 & 0 & 1 & 0 & 0
	\\
	(2+1, 1) 		& 1 & 1 & 0 & 1 & 1 & 0
	\\
	(1 + 1 + 1, 1) 	& 3 & 0 & 0 & 3 & 4 & 1
	\\
	\mbox{Total} &&&&5&5&1
\end{array}
&\quad\quad
\begin{array}{c|c|c|c|c|c|c}
	\multicolumn{7}{c}{\mbox{T{\sc able} 5. }B2\mbox{ Partitions of }5} 
	\\
	& f_{1}^1	& f_{2}^1      & f_{3}^1   	& \omega_1 & \omega_2 		& \omega_3
	\\\hline
	(3+1, 1) 		& 1 & 1 & 0 & 1 & 1 & 0\\
	(2+1+1, 1) 		& 2 & 1 & 0 & 2 & 3 & 1 \\
	(1+1+1+1, 1)	& 4 & 0 & 0 & 4 & 10& 6\\
	\mbox{Total} 			&	& 	&	& 7	& 14&7
\end{array}
\end{align*}

\section{Conjectures and Concluding Remarks}

As demonstrated, Bailey's Lemma is well suited to give inequalities between
the moments of rank-like Lambert series and corresponding crank functions, as 
well as supplying the combinatorial interpretation of the difference of the
symmetrized moments. In our work, we focused our attention to the $a=1$
Bailey pairs of Slater \cite{Slater1,Slater2} with $\alpha_n=\alpha_{-n}$.
We have also seen that it is possible to work with Bailey pairs from other
sources that satisfy the same conditions. However it is not true that all
Bailey pairs, even from Slater's lists, satisfy $\alpha_n=\alpha_{-n}$. We leave
it open to the interested reader to see to what extent one can develop 
similar identities that lead to similar results.

We have studied these new rank moments are far as possible
while handling them in generality. However, we expect each function to possess
interesting properties of its own. In particular, based on how the four
prototypical examples from \cite{Garvan1, JenningsShaffer3} behave, we expect
each ordinary rank moment generating function to correspond to a quasi-mock 
modular form, each ordinary crank moment to correspond to a quasi-modular form, 
and potentially an equation exists relating the certain partial 
derivatives of these rank and crank functions. Using these automorphic properties
one could hope to derive asymptotic formulas for the moments, such as was done
in \cite{BringmannMahlburgRhoades1,BringmannMahlburgRhoades2}.

As we have already seen, a large number of inequalities exist between the 
various moments. We conjecture the following diagram gives all of the inequalities.
Here a downward path from $A_{2k}$ to $B_{2k}$ indicates $A_{2k}(n)\geq B_{2k}(n)$ 
for all positive $k$ and $n$. We have split the diagram into two pieces, to 
decrease the height and handle the large number of crossings.
Besides the rank and crank moments defined in this
article, we also include the overpartition rank $\overline{\eta}_{2k}$, the 
the overpartition M2-rank $\overline{\eta2}_{2k}$, the 
overpartiton M2-crank $\overline{\mu2}_{2k}$, the M2-rank for partitions without
repeated odd parts $\eta2_{2k}$, and the M2-crank for partitions without repeated
odd parts $\mu2_{2k}$ from \cite{JenningsShaffer3}. It is likely some of these
inequalities can be proved using the identities of this article, but to get the
full picture one will need additional techniques. These inequalities have been 
verified for $1\leq k\leq 10$ and $1\leq n< 1000$.

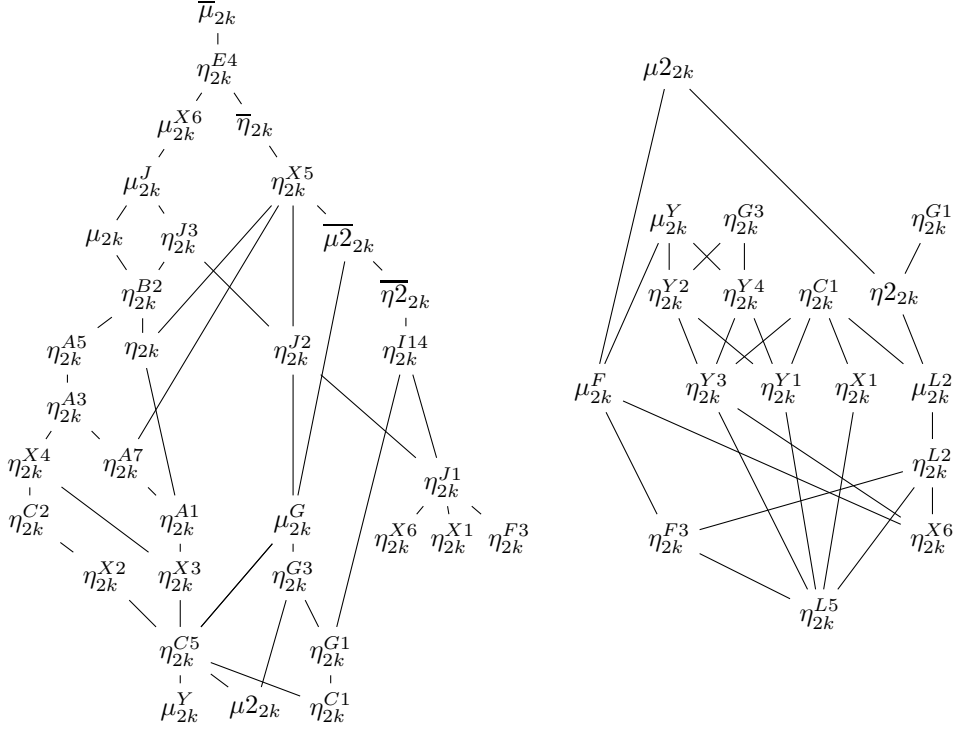
\begin{tikzpicture}[scale=0.25]
	\node (MOverpartition) at (11,38) {$\overline{\mu}_{2k}$};
	\node (NE4) at (11,35) {$\eta^{E4}_{2k}$};
	\node (MX40) at (9,32) {$\mu^{\X{40}}_{2k}$};
	\node (NOverpartition) at (13,32) {$\overline{\eta}_{2k}$};
	\node (MJ) at (7,29) {$\mu^{J}_{2k}$};
	\node (M) at (5,26) {$\mu_{2k}$};
	\node (NJ3) at (9,26) {$\eta^{J3}_{2k}$};
	\node (NB2) at (7,23) {$\eta^{B2}_{2k}$};
	\node (N) at (7,20) {$\eta_{2k}$};
	\node (NX46) at (15,29) {$\eta^{\X{46}}_{2k}$};
	\node (MOverpartitionM2) at (18,26) {$\overline{\mu2}_{2k}$};
	\node (NOverpartitionM2) at (21,23) {$\overline{\eta2}_{2k}$};
	\node (NJ2) at (15,20) {$\eta^{J2}_{2k}$};
	\node (NI14) at (21,20) {$\eta^{I14}_{2k}$};
	\node (NA5) at (3,20) {$\eta^{A5}_{2k}$};
	\node (NA3) at (3,17) {$\eta^{A3}_{2k}$};
	\node (NX42) at (1,14) {$\eta^{\X{42}}_{2k}$};
	\node (NC2) at (1,11) {$\eta^{C2}_{2k}$};
	\node (NA1) at (9,11) {$\eta^{A1}_{2k}$};
	\node (MG) at (15,11) {$\mu^{G}_{2k}$};
	\node (NG3) at (15,8) {$\eta^{G3}_{2k}$};
	\node (NA7) at (6,14) {$\eta^{A7}_{2k}$};
	\node (NX41) at (9,8) {$\eta^{\X{41}}_{2k}$};
	\node (NC5) at (9,4) {$\eta^{C5}_{2k}$};
	\node (NG1) at (17,4) {$\eta^{G1}_{2k}$};
	\node (MY) at (9,1) {$\mu^{Y}_{2k}$};
	\node (M2) at (13,1) {$\mu 2_{2k}$};
	\node (NC1) at (17,1) {$\eta^{C1}_{2k}$};
	\node (NX39) at (5,8) {$\eta^{\X{39}}_{2k}$};
	\node (NJ1) at (23,13) {$\eta^{J1}_{2k}$};
	\node (NX40) at (20.5,10) {$\eta^{\X{40}}_{2k}$};
	\node (NX38) at (23.5,10) {$\eta^{\X{38}}_{2k}$};
	\node (NF3) at (26.5,10) {$\eta^{F3}_{2k}$};

	\draw (MOverpartition) -- (NE4);
	\draw (NE4) -- (MX40);
	\draw (NE4) -- (NOverpartition);
	\draw (MX40) -- (MJ);
	\draw (MJ) -- (M);
	\draw (MJ) -- (NJ3);
	\draw (M) -- (NB2);
	\draw (NJ3) -- (NB2);
	\draw (NB2) -- (N);
	\draw (NOverpartition) -- (NX46);
	\draw (NX46) -- (MOverpartitionM2);
	\draw (MOverpartitionM2) -- (NOverpartitionM2);
	\draw (MOverpartitionM2) -- (MG);
	\draw (NX46) -- (N);
	\draw (NX46) -- (NJ2);
	\draw (NOverpartitionM2) -- (NI14);
	\draw (NJ3) -- (NJ2);
	\draw (NJ2) -- (NJ1);
	\draw (NI14) -- (NJ1);
	\draw (NJ1) -- (NX40);
	\draw (NJ2) -- (MG);
	\draw (MG) -- (NG3);
	\draw (MG) -- (NC5);
	\draw (NI14) -- (NG1);
	\draw (NG3) -- (NG1);
	\draw (NG1) -- (NC1);
	\draw (NG3) -- (M2);
	\draw (NB2) -- (NA5);
	\draw (NA5) -- (NA3);
	\draw (NA3) -- (NX42);
	\draw (NA3) -- (NA7);
	\draw (NX42) -- (NC2);
	\draw (NA7) -- (NA1);
	\draw (N) -- (NA1);
	\draw (NA1) -- (NX41);
	\draw (NX42) -- (NX41);
	\draw (NC2) -- (NX39);
	\draw (NX39) -- (NC5);
	\draw (NX41) -- (NC5);
	\draw (MG) -- (NC5);
	\draw (NC5) -- (MY);
	\draw (NC5) -- (M2);
	\draw (NC5) -- (NC1);
	\draw (NX46) -- (NA7);
	\draw (NJ1) -- (NX38);
	\draw (NJ1) -- (NF3);

	\node (AMY) at 	(5+\xOffSetForTikzPicture,22+\yOffSetForTikzPicture) {$\mu^{Y}_{2k}$};
	\node (ANG3) at (9+\xOffSetForTikzPicture,22+\yOffSetForTikzPicture) {$\eta^{G3}_{2k}$};
	\node (AM2) at (5+\xOffSetForTikzPicture,30+\yOffSetForTikzPicture) {$\mu2_{2k}$};
	\node (ANG1) at (19+\xOffSetForTikzPicture,22+\yOffSetForTikzPicture) {$\eta^{G1}_{2k}$};
	\node (AMF) at (1+\xOffSetForTikzPicture,13+\yOffSetForTikzPicture) {$\mu^{F}_{2k}$};
	\node (ANY2) at (5+\xOffSetForTikzPicture,18+\yOffSetForTikzPicture) {$\eta^{Y2}_{2k}$};
	\node (ANY4) at (9+\xOffSetForTikzPicture,18+\yOffSetForTikzPicture) {$\eta^{Y4}_{2k}$};
	\node (ANC1) at (13+\xOffSetForTikzPicture,18+\yOffSetForTikzPicture) {$\eta^{C1}_{2k}$};
	\node (AN2) at (17+\xOffSetForTikzPicture,18+\yOffSetForTikzPicture) {$\eta2_{2k}$};
	\node (ANY3) at (7+\xOffSetForTikzPicture,13+\yOffSetForTikzPicture) {$\eta^{Y3}_{2k}$};
	\node (ANY1) at (11+\xOffSetForTikzPicture,13+\yOffSetForTikzPicture) {$\eta^{Y1}_{2k}$};
	\node (AML2) at (19+\xOffSetForTikzPicture,13+\yOffSetForTikzPicture) {$\mu^{L2}_{2k}$};
	\node (ANX40) at (19+\xOffSetForTikzPicture,5+\yOffSetForTikzPicture) {$\eta^{\X{40}}_{2k}$};
	\node (ANL2) at (19+\xOffSetForTikzPicture,9+\yOffSetForTikzPicture) {$\eta^{L2}_{2k}$};
	\node (ANX38) at (15+\xOffSetForTikzPicture,13+\yOffSetForTikzPicture) {$\eta^{\X{38}}_{2k}$};
	\node (ANF3) at (5+\xOffSetForTikzPicture,5+\yOffSetForTikzPicture) {$\eta^{F3}_{2k}$};
	\node (ANL5) at (13+\xOffSetForTikzPicture,1+\yOffSetForTikzPicture) {$\eta^{L5}_{2k}$};

	\draw (AMY) -- (AMF);
	\draw (AMY) -- (ANY2);
	\draw (AMY) -- (ANY4);
	\draw (ANG3) -- (ANY4);
	\draw (AM2) -- (AN2);
	\draw (ANG1) -- (AN2);
	\draw (AMF) -- (ANX40);
	\draw (ANY2) -- (ANY3);
	\draw (ANY2) -- (ANY1);
	\draw (ANY4) -- (ANY3);
	\draw (ANY4) -- (ANY1);
	\draw (ANC1) -- (ANY3);
	\draw (ANC1) -- (ANY1);
	\draw (ANC1) -- (ANX38);
	\draw (ANC1) -- (AML2);
	\draw (AN2) -- (AML2);
	\draw (ANY3) -- (ANX40);
	\draw (AML2) -- (ANL2);
	\draw (ANL2) -- (ANX40);
	\draw (AM2) -- (AMF);
	\draw (ANG3) -- (ANY2);
	\draw (AMF) --	(ANF3);
	\draw (ANL2) --	(ANF3);
	\draw (ANL2) --	(ANL5);
	\draw (ANY1) --	(ANL5);
	\draw (ANY3) --	(ANL5);
	\draw (ANX38) -- (ANL5);
	\draw (ANF3) --	(ANL5);
\end{tikzpicture}

Additionally based on numerical evidence, it would appear that some of the 
higher order spt functions satisfy a number of congruences. While a few of 
these may follow by elementary means, to approach these likely one should 
start with the automorphic properties of the moments mentioned above, as this 
is the method that worked for the original examples.
We conjecture the following congruences,
\begin{align*}
	0
	&\equiv	
	\hspt{B2}{3}{4n+3}
	\equiv	
	\hspt{E4}{2}{31n}
	\equiv	
	\hspt{E4}{2}{41n}
	\equiv	
	\hspt{E4}{2}{47n}
	\equiv	
	\hspt{E4}{2}{16n+1}
	\equiv	
	\hspt{E4}{2}{32n+2}
	\\&\equiv	
	\hspt{E4}{2}{8n+7}
	\equiv	
	\hspt{E4}{2}{49n+7}
	\equiv	
	\hspt{E4}{2}{49n+14}
	\equiv	
	\hspt{E4}{2}{18n+15}
	\equiv	
	\hspt{E4}{2}{24n+17}
	\\&\equiv	
	\hspt{E4}{2}{40n+17}
	\equiv	
	\hspt{E4}{2}{36n+21}
	\equiv	
	\hspt{E4}{2}{45n+21}
	\equiv	
	\hspt{E4}{2}{49n+21}
	\equiv	
	\hspt{E4}{2}{32n+28}
	\\&\equiv	
	\hspt{E4}{2}{49n+28}
	\equiv	
	\hspt{E4}{2}{36n+30}
	\equiv	
	\hspt{E4}{2}{40n+33}
	\equiv	
	\hspt{E4}{2}{45n+33}
	\equiv	
	\hspt{E4}{2}{48n+34}
	\\&\equiv	
	\hspt{E4}{2}{49n+35}
	\equiv	
	\hspt{E4}{2}{45n+39}
	\equiv	
	\hspt{E4}{2}{45n+42}
	\equiv	
	\hspt{E4}{2}{49n+42}
	\equiv	
	\hspt{E4}{3}{16n+3}
	\\&\equiv	
	\hspt{E4}{3}{32n+10}
	\equiv	
	\hspt{E4}{3}{16n+13}
	\equiv	
	\hspt{E4}{3}{49n+21}
	\equiv	
	\hspt{E4}{3}{32n+22}
	\equiv	
	\hspt{E4}{3}{49n+35}
	\\&\equiv	
	\hspt{E4}{3}{49n+42}
	\equiv	
	\hspt{E4}{4}{31n}
	\equiv	
	\hspt{E4}{4}{47n}
	\equiv	
	\hspt{E4}{4}{32n+1}
	\equiv	
	\hspt{E4}{4}{16n+7}
	\equiv	
	\hspt{E4}{4}{32n+14}
	\\&\equiv	
	\hspt{E4}{4}{48n+17}
	\equiv	
	\hspt{E4}{5}{32n+11}
	\equiv	
	\hspt{E4}{5}{32n+13}
	\equiv	
	\hspt{E4}{6}{31n}
	\equiv	
	\hspt{E4}{6}{32n+17}
	\\&\equiv	
	\hspt{E4}{6}{32n+23}
	\equiv	
	\hspt{E4}{7}{32n+29}
	\equiv	
	\hspt{E4}{8}{31n}
	\equiv	
	\hspt{E4}{8}{32n+7}
	\equiv	
	\hspt{\X{41}}{3}{4n+2}
	\equiv	
	\hspt{\X{42}}{3}{44n+28}
	\\&\equiv
	\hspt{\X{46}}{2}{48n+1}
	\equiv	
	\hspt{\X{46}}{2}{24n+7}
	\equiv	
	\hspt{\X{46}}{2}{48n+14}
	\equiv	
	\hspt{\X{46}}{2}{24n+17}
	\equiv	
	\hspt{\X{46}}{2}{24n+23}
	\\&\equiv	
	\hspt{\X{46}}{2}{48n+34}
	\equiv	
	\hspt{\X{46}}{2}{48n+46}
	\equiv	
	\hspt{\X{46}}{3}{48n+11}
	\equiv	
	\hspt{\X{46}}{3}{48n+13}
	\equiv	
	\hspt{\X{46}}{3}{48n+29}
	\\&\equiv	
	\hspt{\X{46}}{3}{48n+43}
	\equiv	
	\hspt{\X{46}}{4}{48n+7}
	\equiv	
	\hspt{\X{46}}{4}{48n+17}
	\equiv	
	\hspt{\X{46}}{4}{32n+23}
	\equiv	
	\hspt{\X{46}}{4}{48n+23}
	\\&\equiv
	\hspt{\X{40}}{2}{46n+3}
	\equiv	
	\hspt{\X{40}}{2}{49n+15}
	\equiv	
	\hspt{\X{40}}{2}{49n+29}
	\equiv	
	\hspt{\X{40}}{2}{41n+36}
	\equiv	
	\hspt{\X{40}}{2}{49n+36}
	\\&\equiv	
	\hspt{\X{40}}{3}{4n}
	\equiv	
	\hspt{\X{40}}{3}{18n+2}	
	\equiv	
	\hspt{\X{40}}{3}{18n+14}
	\equiv	
	\hspt{\X{40}}{3}{49n+15}
	\equiv	
	\hspt{\X{40}}{3}{49n+29}
	\\&\equiv	
	\hspt{\X{40}}{3}{49n+36}
	\equiv	
	\hspt{\X{40}}{6}{8n}
	\equiv	
	\hspt{\X{40}}{6}{8n+7}
	\equiv	
	\hspt{\X{40}}{6}{41n+36}
	\equiv		
	\hspt{F3}{2}{4n+1}
	\equiv	
	\hspt{F3}{3}{4n}
	\\&\equiv	
	\hspt{F3}{4}{8n+1}
	\equiv	
	\hspt{F3}{5}{8n+4}
	\equiv	
	\hspt{F3}{6}{8n+5}
	\equiv	
	\hspt{F3}{8}{16n+1}
	\equiv	
	\hspt{F3}{10}{16n+5}
	\pmod{2}
	,\\
	0
	&\equiv
	\hspt{A3}{2}{9n}
	\equiv	
	\hspt{B2}{4}{3n}
	\equiv	
	\hspt{E4}{3}{27n+15}
	\equiv	
	\hspt{E4}{4}{27n+6}
	\equiv	
	\hspt{\X{46}}{2}{9n}
	\equiv	
	\hspt{\X{46}}{2}{24n+11}
	\\&\equiv	
	\hspt{\X{46}}{2}{32n+12}
	\equiv	
	\hspt{\X{46}}{8}{27n}
	\equiv	
	\hspt{F3}{2}{6n+1}
	\equiv	
	\hspt{L2}{3}{27n+26}
	\pmod{3}
	,\\
	0
	&\equiv
	\hspt{L5}{4}{44n+28}
	\equiv	
	\hspt{E4}{2}{16n+14}
	\equiv	
	\hspt{E4}{2}{36n+33}
	\pmod{4}
	,\\
	0
	&\equiv
	\hspt{A1}{2}{5n}
	\equiv	
	\hspt{A1}{2}{5n+1}
	\equiv	
	\hspt{A1}{4}{25n+24}
	\equiv	
	\hspt{A1}{5}{25n+24}
	\equiv	
	\hspt{A3}{2}{5n+1}
	\equiv	
	\hspt{A3}{2}{5n+2}
	\\&\equiv	
	\hspt{A3}{2}{5n+4}
	\equiv	
	\hspt{A5}{2}{5n}
	\equiv	
	\hspt{A5}{2}{5n+4}
	\equiv	
	\hspt{A5}{3}{25n+9}
	\equiv	
	\hspt{A5}{3}{25n+14}
	\equiv	
	\hspt{A7}{2}{5n+1}
	\\&\equiv	
	\hspt{A7}{2}{5n+4}
	\equiv	
	\hspt{A7}{3}{25n+24}
	\equiv	
	\hspt{B2}{2}{5n+1}
	\equiv	
	\hspt{B2}{2}{5n+2}
	\equiv	
	\hspt{B2}{2}{5n+4}
	\equiv	
	\hspt{B2}{3}{25n+1}
	\\&\equiv	
	\hspt{B2}{3}{25n+9}
	\equiv	
	\hspt{B2}{4}{25n+2}
	\equiv	
	\hspt{B2}{4}{25n+4}
	\equiv	
	\hspt{B2}{8}{25n+12}
	\equiv	
	\hspt{C5}{2}{5n}
	\equiv	
	\hspt{C5}{2}{5n+1}
	\\&\equiv	
	\hspt{C5}{2}{5n+4}
	\equiv	
	\hspt{C5}{3}{25n+24}
	\equiv	
	\hspt{C5}{4}{25n+3}
	\equiv	
	\hspt{C5}{4}{25n+5}
	\equiv	
	\hspt{C5}{4}{25n+24}
	\\&\equiv	
	\hspt{C5}{5}{25n+4}
	\equiv	
	\hspt{C5}{5}{25n+24}
	\equiv	
	\hspt{C5}{6}{25n+5}
	\equiv	
	\hspt{C5}{6}{25n+10}
	\equiv	
	\hspt{E4}{2}{5n}
	\equiv	
	\hspt{E4}{2}{5n+2}
	\\&\equiv	
	\hspt{F3}{3}{25n+3}
	\equiv	
	\hspt{F3}{3}{25n+23}
	\equiv	
	\hspt{Y1}{5k}{10n+3}
	\equiv	
	\hspt{Y1}{5}{25n+8}
	\equiv	
	\hspt{Y1}{10}{25n+8}
	\\&\equiv	
	\hspt{Y2}{5k}{10n+9}
	\pmod{5}
	,\\
	0
	&\equiv
	\hspt{A1}{2}{49n+12}
	\equiv	
	\hspt{A1}{3}{49n+47}
	\equiv	
	\hspt{A3}{2}{49n+19}
	\equiv	
	\hspt{A5}{2}{7n+1}
	\equiv	
	\hspt{A5}{3}{7n}
	\equiv	
	\hspt{A5}{3}{7n+1}
	\\&\equiv	
	\hspt{A5}{3}{7n+3}
	\equiv	
	\hspt{A5}{3}{7n+5}
	\equiv	
	\hspt{A5}{6}{7n+5}
	\equiv	
	\hspt{A7}{2}{7n}
	\equiv	
	\hspt{A7}{2}{7n+1}
	\equiv	
	\hspt{A7}{3}{7n}
	\\&\equiv	
	\hspt{A7}{3}{7n+1}
	\equiv	
	\hspt{A7}{3}{7n+2}
	\equiv	
	\hspt{A7}{3}{7n+4}
	\equiv	
	\hspt{A7}{5}{49n+47}
	\equiv	
	\hspt{A7}{6}{49n+47}
	\\&\equiv	
	\hspt{B2}{2}{7n+1}
	\equiv	
	\hspt{B2}{2}{7n+5}
	\equiv	
	\hspt{B2}{3}{7n}
	\equiv	
	\hspt{B2}{3}{7n+1}
	\equiv	
	\hspt{B2}{3}{7n+3}
	\equiv	
	\hspt{B2}{3}{7n+5}
	\\&\equiv	
	\hspt{B2}{4}{49n+1}
	\equiv	
	\hspt{B2}{5}{49n+33}
	\equiv	
	\hspt{B2}{6}{7n+5}
	\pmod{7}	
	,\\
	0
	&\equiv
	\hspt{E4}{2}{32n+30} \pmod{8}
	,\\
	0
	&\equiv	
	\hspt{\X{46}}{5}{27n}
	\equiv
	\hspt{\X{46}}{5}{27n+18}
	\pmod{9}
	,\\
	0
	&\equiv
	\hspt{B2}{2}{11n+1}
	\pmod{11}
	,\\
	0
	&\equiv	
	\hspt{A5}{3}{25n+24}
	\equiv
	\hspt{B2}{2}{25n+14}
	\equiv
	\hspt{C5}{2}{25n+24} 
	\equiv
	\hspt{F3}{2}{25n+23}
	\pmod{25}
	,\\
	0
	&\equiv
	\hspt{B2}{3}{49n+26}
	\pmod{49}
	.
\end{align*}

Lastly, there is also the 
concept of positive moments, where $m$ ranges over just the positive integers,
rather than all of $\mathbb{Z}$, that is to say,
\begin{align*}
	N^+_k(n)
	&= 
		\sum_{m=1}^{\infty}m^k N(m,n)
	,&
	\eta^+_{k}(n)
	&= 
		\sum_{m=1}^{\infty}
		\binom{m+\lfloor\frac{k-1}{2}\rfloor}{k}N(m,n)
	,\\	
	M^+_k(n)
	&= 
		\sum_{m=1}^{\infty}m^kM(m,n)
	,&
	\mu^+_{k}(n)
	&= 
		\sum_{m=1}^{\infty}
		\binom{m+\lfloor\frac{k-1}{2}\rfloor}{k}M(m,n)
	.
\end{align*}
The advantage to these positive moments is that while 
$N_{2k}(n)=2N^{+}_{2k}(n)$ and $M_{2k}(n)=2M^{+}_{2k}(n)$, it is no
longer the case that the odd moments are zero. It is true that
$M^{+}_{2k+1}(n)>N^{+}_{2k+1}(n)$, and there have been several studies of
the positive moments corresponding to the rank and crank of ordinary
partitions as well as overpartitions, see
\cite{AndrewsChanKimOsburn1, AndrewsChanKim1, BringmannMahlburg1, 
LarsenRustSwisher1, ZapataRolon}.
As such we should expect that analogous results and inequalities hold for the 
moments of this article, however our methods do not directly apply
and it is not clear if one can handle positive moments in the generality we have
managed for the original moments.

\bibliographystyle{abbrv}
\bibliography{higherOrderSPTBaileyPairsRef}

\begin{thebibliography}{10}

\bibitem{AndrewsChanKimOsburn1}
G.~Andrews, S.~H. Chan, B.~Kim, and R.~Osburn.
\newblock The first positive rank and crank moments for overpartitions.
\newblock {\em Ann. Comb.}, 20(2):193--207, 2016.

\bibitem{Andrews1}
G.~E. Andrews.
\newblock {\em The theory of partitions}.
\newblock Addison-Wesley Publishing Co., Reading, Mass.-London-Amsterdam, 1976.
\newblock Encyclopedia of Mathematics and its Applications, Vol. 2.

\bibitem{Andrews2}
G.~E. Andrews.
\newblock {\em {$q$}-series: their development and application in analysis,
  number theory, combinatorics, physics, and computer algebra}, volume~66 of
  {\em CBMS Regional Conference Series in Mathematics}.
\newblock Published for the Conference Board of the Mathematical Sciences,
  Washington, DC; by the American Mathematical Society, Providence, RI, 1986.

\bibitem{Andrews4}
G.~E. Andrews.
\newblock Partitions, {D}urfee symbols, and the {A}tkin-{G}arvan moments of
  ranks.
\newblock {\em Invent. Math.}, 169(1):37--73, 2007.

\bibitem{Andrews3}
G.~E. Andrews.
\newblock The number of smallest parts in the partitions of {$n$}.
\newblock {\em J. Reine Angew. Math.}, 624:133--142, 2008.

\bibitem{AndrewsBaxter1}
G.~E. Andrews and R.~J. Baxter.
\newblock A motivated proof of the {R}ogers-{R}amanujan identities.
\newblock {\em Amer. Math. Monthly}, 96(5):401--409, 1989.

\bibitem{AndrewsBerndt1}
G.~E. Andrews and B.~C. Berndt.
\newblock {\em Ramanujan's lost notebook. {P}art {I}}.
\newblock Springer, New York, 2005.

\bibitem{AndrewsChanKim1}
G.~E. Andrews, S.~H. Chan, and B.~Kim.
\newblock The odd moments of ranks and cranks.
\newblock {\em J. Combin. Theory Ser. A}, 120(1):77--91, 2013.

\bibitem{AndrewsDixitYee1}
G.~E. Andrews, A.~Dixit, and A.~J. Yee.
\newblock Partitions associated with the {R}amanujan/{W}atson mock theta
  functions {$\omega(q)$}, {$\nu(q)$} and {$\phi(q)$}.
\newblock {\em Res. Number Theory}, 1:Art. 19, 25, 2015.

\bibitem{AndrewsGarvanLiang1}
G.~E. Andrews, F.~G. Garvan, and J.~Liang.
\newblock Combinatorial interpretations of congruences for the spt-function.
\newblock {\em Ramanujan J.}, 29(1-3):321--338, 2012.

\bibitem{Askey1}
R.~Askey.
\newblock Orthogonal polynomials and theta functions.
\newblock In {\em Theta functions---{B}owdoin 1987, {P}art 2 ({B}runswick,
  {ME}, 1987)}, volume~49 of {\em Proc. Sympos. Pure Math.}, pages 299--321.
  Amer. Math. Soc., Providence, RI, 1989.

\bibitem{AtkinGarvan1}
A.~O.~L. Atkin and F.~G. Garvan.
\newblock Relations between the ranks and cranks of partitions.
\newblock {\em Ramanujan J.}, 7(1-3):343--366, 2003.
\newblock Rankin memorial issues.

\bibitem{Bailey1}
W.~N. Bailey.
\newblock Some identities in combinatory analysis.
\newblock {\em Proc. London Math. Soc. (2)}, 49:421--425, 1947.

\bibitem{Bailey2}
W.~N. Bailey.
\newblock Identities of the {R}ogers-{R}amanujan type.
\newblock {\em Proc. London Math. Soc. (2)}, 50:1--10, 1948.

\bibitem{Berndt1}
B.~C. Berndt.
\newblock {\em Ramanujan's notebooks. {P}art {III}}.
\newblock Springer-Verlag, New York, 1991.

\bibitem{BowmannMclaughlinSills1}
D.~Bowman, J.~Mc~Laughlin, and A.~V. Sills.
\newblock Some more identities of the {R}ogers-{R}amanujan type.
\newblock {\em Ramanujan J.}, 18(3):307--325, 2009.

\bibitem{BringmannLovejoyOsburn1}
K.~Bringmann, J.~Lovejoy, and R.~Osburn.
\newblock Rank and crank moments for overpartitions.
\newblock {\em J. Number Theory}, 129(7):1758--1772, 2009.

\bibitem{BringmannMahlburg1}
K.~Bringmann and K.~Mahlburg.
\newblock Asymptotic inequalities for positive crank and rank moments.
\newblock {\em Trans. Amer. Math. Soc.}, 366(2):1073--1094, 2014.

\bibitem{BringmannMahlburgRhoades1}
K.~Bringmann, K.~Mahlburg, and R.~C. Rhoades.
\newblock Asymptotics for rank and crank moments.
\newblock {\em Bull. Lond. Math. Soc.}, 43(4):661--672, 2011.

\bibitem{BringmannMahlburgRhoades2}
K.~Bringmann, K.~Mahlburg, and R.~C. Rhoades.
\newblock Taylor coefficients of mock-{J}acobi forms and moments of partition
  statistics.
\newblock {\em Math. Proc. Cambridge Philos. Soc.}, 157(2):231--251, 2014.

\bibitem{Garvan1}
F.~G. Garvan.
\newblock Higher order spt-functions.
\newblock {\em Adv. Math.}, 228(1):241--265, 2011.

\bibitem{GarvanJenningsShaffer1}
F.~G. Garvan and C.~Jennings-Shaffer.
\newblock The spt-crank for overpartitions.
\newblock {\em Acta Arith.}, 166(2):141--188, 2014.

\bibitem{GarvanJenningsShaffer2}
F.~G. Garvan and C.~Jennings-Shaffer.
\newblock Exotic {B}ailey--{S}later {SPT}-functions {II}:
  {H}ecke--{R}ogers-type double sums and {B}ailey pairs from groups {A}, {C},
  {E}.
\newblock {\em Adv. Math.}, 299:605--639, 2016.

\bibitem{GriffinOnoOle1}
M.~J. Griffin, K.~Ono, and S.~O. Warnaar.
\newblock A framework of {R}ogers-{R}amanujan identities and their arithmetic
  properties.
\newblock {\em Duke Math. J.}, 165(8):1475--1527, 2016.

\bibitem{JenningsShaffer3}
C.~Jennings-Shaffer.
\newblock Higher order {SPT} functions for overpartitions, overpartitions with
  smallest part even, and partitions with smallest part even and without
  repeated odd parts.
\newblock {\em J. Number Theory}, 149:285--312, 2015.

\bibitem{JenningsShaffer2}
C.~Jennings-Shaffer.
\newblock {Exotic Bailey-Slater SPT-Functions III: Bailey Pairs From Groups B,
  F, G, and J}.
\newblock {\em Acta Arith.}, 173(4):317--364, 2016.

\bibitem{JenningsShaffer1}
C.~Jennings-Shaffer.
\newblock {Exotic Bailey–Slater SPT-functions I: Group A}.
\newblock {\em Advances in Mathematics}, 305:479 -- 514, 2017.

\bibitem{LarsenRustSwisher1}
A.~Larsen, A.~Rust, and H.~Swisher.
\newblock Inequalities for positive rank and crank moments of overpartitions.
\newblock {\em Int. J. Number Theory}, 10(8):2115--2133, 2014.

\bibitem{LepowskyWilson1}
J.~Lepowsky and R.~L. Wilson.
\newblock A {L}ie theoretic interpretation and proof of the
  {R}ogers-{R}amanujan identities.
\newblock {\em Adv. in Math.}, 45(1):21--72, 1982.

\bibitem{Patkowski1}
A.~E. Patkowski.
\newblock Another smallest part function related to {A}ndrews' {${\rm spt}$}
  function.
\newblock {\em Acta Arith.}, 168(2):101--105, 2015.

\bibitem{Patkowski2}
A.~E. Patkowski.
\newblock A strange partition theorem related to the second {A}tkin-{G}arvan
  moment.
\newblock {\em Int. J. Number Theory}, 11(7):2191--2197, 2015.

\bibitem{Rogers1}
L.~J. Rogers.
\newblock Second memoir on the expansion of certain infinite products.
\newblock {\em Proc. London Math. Soc.}, 25:318--343, 1894.

\bibitem{Slater1}
L.~J. Slater.
\newblock A new proof of {R}ogers's transformations of infinite series.
\newblock {\em Proc. London Math. Soc. (2)}, 53:460--475, 1951.

\bibitem{Slater2}
L.~J. Slater.
\newblock Further identities of the {R}ogers-{R}amanujan type.
\newblock {\em Proc. London Math. Soc. (2)}, 54:147--167, 1952.

\bibitem{ZapataRolon}
J.~M. Zapata~Rolon.
\newblock Asymptotics of higher order ospt-functions for overpartitions.
\newblock {\em Ann. Comb.}, 20(1):177--191, 2016.

\end{thebibliography}

\end{document}